\documentclass[reqno]{amsart}
\usepackage{graphicx,amsmath,amsthm,amsfonts,amssymb,mathrsfs}
\usepackage[usenames,dvipsnames]{xcolor}
\usepackage{hyperref}
\hypersetup{%
	colorlinks=true, linkcolor=blue,
	citecolor=ForestGreen
}
\usepackage[paper=letterpaper,margin=.8in]{geometry}
\usepackage[inline]{enumitem} 

\theoremstyle{plain}
\newtheorem{thm}{Theorem}[section]

\newtheorem{cor}[thm]{Corollary}
\newtheorem{prop}[thm]{Proposition}
\newtheorem{lem}[thm]{Lemma}
\theoremstyle{definition}
\newtheorem{defn}[thm]{Definition}
\newtheorem{rmk}[thm]{Remark}
\numberwithin{equation}{section}
\allowdisplaybreaks

\usepackage{acronym}
\acrodef{S6V}{Stochastic Six Vertex}
\acrodef{SPDE}{Stochastic Partial Differential Equation}
\acrodef{STE}{Stochastic Telegraph Equation}
\acrodef{TE}{Telegraph Equation}
\acrodef{CLT}{Central Limit Theorem}
\acrodef{KPZ}{Kardar--Parisi--Zhang}

\newcommand{\Z}{\mathbb{Z}}

\newcommand{\R}{\mathbb{R}}
\newcommand{\Rp}{\mathbb{R}_+}
\newcommand{\Lip}{\text{Lip}}
\newcommand{\h}{\mathbf h}

\newcommand{\Ex}{\mathbf{E}}		
\newcommand{\ind}{\mathbf{1}}		
\newcommand{\set}[1]{{\{#1\}}}		
\newcommand{\norm}[1]{\Vert#1\Vert}

\newcommand{\Rie}{\mathcal{R}}
\newcommand{\rie}{\mathcal{R}^\text{d}}
\newcommand{\RieI}{\widetilde{\mathcal{R}}}				
\newcommand{\rieI}{\widetilde{\mathcal{R}}^\text{d}}	

\newcommand{\dnoise}{\xi^\text{d}}
\newcommand{\noise}{\xi}
\newcommand{\img}{\mathbf{i}}
\newcommand{\barx}{\overline{\mathbf{x}}}
\newcommand{\bary}{\overline{\mathbf{y}}}
\newcommand{\x}{\mathbf{x}}
\newcommand{\y}{\mathbf{y}}
\newcommand{\q}{\mathfrak{q}}
\newcommand{\filtM}{\mathscr{F}} 	
\newcommand{\filtnoi}{\mathscr{G}} 	
\newcommand{\QV}{Q} 				
\newcommand{\D}{D}					
\newcommand{\Dd}{D^\text{d}}		
\newcommand{\Ddd}{\hat{D}^\text{d}}	
\newcommand{\bk}{B}					
\newcommand{\bdy}{F}				
\newcommand{\Av}{\Sigma^*_L}		
\newcommand{\generic}{\mathcal{K}}
\newcommand{\Generic}{\mathscr{K}}
\newcommand{\I}{\mathcal{I}}
\newcommand{\genericd}{\mathcal{K}^\text{d}}
\newcommand{\Genericd}{\mathscr{K}^\text{d}}
\newcommand{\Id}{\mathcal{I}^\text{d}}

\newcommand{\e}{\varepsilon}
\newcommand{\barphi}{\bar{\phi}}
\newcommand{\V}{V}
\newcommand{\Vd}{\mathsf{v}}

\renewcommand{\bar}[1]{\overline{#1}}

\renewcommand{\hat}[1]{\widehat{#1}}
\renewcommand{\tilde}[1]{\widetilde{#1}}

\usepackage{graphicx}
\newcommand*{\Cdot}{{\raisebox{-0.5ex}{\scalebox{1.8}{$\cdot$}}}} 

\begin{document}
\title[STE limit for S6V model]{Stochastic Telegraph Equation Limit\\{}for the Stochastic Six Vertex Model}
\author[H.\ Shen]{Hao Shen}
\address{H.\ Shen,
	Departments of Mathematics, Columbia University,
	\newline\hphantom{\hspace{15pt}H.\ Shen}
	2990 Broadway, New York, NY 10027}
\email{pkushenhao@gmail.com}
\author[L-C.\ Tsai]{Li-Cheng Tsai}
\address{L-C.\ Tsai,
	Departments of Mathematics, Columbia University,
	\newline\hphantom{\hspace{15pt}L-C.\ Tsai}
	2990 Broadway, New York, NY 10027}
\email{lctsai.math@gmail.com}
\maketitle

\begin{abstract}
In this article we study the stochastic six vertex model under the scaling proposed by \cite{borodin18}, 
where the weights of corner-shape vertices are tuned to zero,
and prove \cite[Conjecture 6.1]{borodin18}:
that the height fluctuation converges in finite dimensional distributions to the solution of stochastic telegraph equation.
\end{abstract}

\section{Introduction}
The six vertex model is a model of tiling on subset $ \Omega $ of $ \Z^2 $,
with each site $ (x,y)\in\Omega $ being tiled with of the six types as depicted in Figure~\ref{tbl:S6VWeights}.
The tiling obeys the rule that each (solid) line connects to a neighboring line.
See Figure~\ref{fig:6v} for a generic realization for tiling.
In this article we focus on the stochastic weight, with $ b_1,b_2\in(0,1) $, as depicted Figure~\ref{tbl:S6VWeights},
and take the domain $ \Omega := \Z^2_{\geq 0} $ to be the first quadrant.
Fix boundary conditions on the axises $ \Z_{\geq 0}\times\set{0} $ and $ \set{0}\times\Z_{\geq 0} $
that indicate whether a given site along the axises has a line entering into $ \Z_{\geq 0}^2 $.
Starting from the site $ (1,1) $, we tile the given site with one of the six vertices
with reference to the incoming (bottom and left) line configurations, and with probability given by the weights.
This tiling construction then progresses sequentially in the linear order
$
	(1,1), (2,1), (1,2), (3,1), (2,2), (1,3), \ldots
$
to the entire quadrant.
For a given tiling there associates a height function, $ H(x,y) $.
This is a $ \Z_{\geq 0} $-valued function defined on $ (x,y)\in\Z_{\geq 0}^2 $,
so that, once interpreting a given tiling as non-intersecting lines,
the level set of $ H(x,y) $, $ x,y\in\Z_{\geq 0} $ are exactly these non-intersecting lines,
with the convention $ H(0,0):= 0 $.
See Figure~\ref{fig:6v}.

\begin{figure}[ht]
\centering
\begin{tabular}{|c|c|c|c|c|c|}
	\hline
	$I$ & $II$ & $III$ & $IV$ & $V$ & $VI$ \\
	\hline
	\includegraphics[width=35pt]{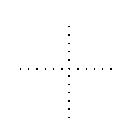}
	&
	\includegraphics[width=35pt]{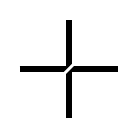}
	&
	\includegraphics[width=35pt]{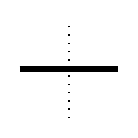}
	&
	\includegraphics[width=35pt]{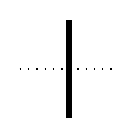}
	&
	\includegraphics[width=35pt]{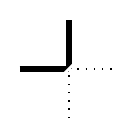}
	&
	\includegraphics[width=35pt]{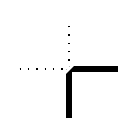}
	\\
	\hline
	$1$ & $1$ & $b_1$ & $b_2$ & $1-b_1$  & $1-b_2 $
	\\
	\hline
\end{tabular}
\caption{Six vertices with their weights.}
\label{tbl:S6VWeights}
\end{figure}

Initiated in \cite{GS92}, the \ac{S6V} model has caught much attention.
Being a special case of the special case the six vertex model,
it describes phenomena in equilibrium statistical mechanics.
On the other hand, the \ac{S6V} model also connects to nonequilibrium growth phenomena within the \ac{KPZ} universality class.
In particular, \cite{BCG2016} proved that, starting with step initial condition,
the height fluctuation converges at one-point to GUE Tracy--Widom distribution.
One point convergence under different initial condition (including the stationary case)
was obtained in \cite{AB16,A16aa},
and \cite{BBCW} studied a half-space version of the \ac{S6V} model 
and demonstrated that its one-point asymptotics match the prediction from other models in the \ac{KPZ} class.
In a related but slightly different direction,
there has been study where one tunes the weights simultaneously with spacetime scaling
in order to observe \ac{SPDE} limit.
In, \cite{BO16} it is showed that under a certain tuning of the weights,
one point distribution of the \ac{S6V} model converges to that of the \ac{KPZ} equation.
For a higher-spin generalization of the \ac{S6V} model (see \cite{Corwin2016,BP16a}),
\cite{CT15} obtained a microscopic Hopf--Cole transform, and showed convergence to \ac{KPZ} equation at process level.
For \ac{S6V} under the scaling $ b_1/b_2\to1 $, $ b_1,b_2\to b\in(0,1) $,
the convergence to \ac{KPZ} equation was obtain in \cite{CGST18} via a Markov duality method.

\begin{figure}[h]
\includegraphics[width=.6\linewidth]{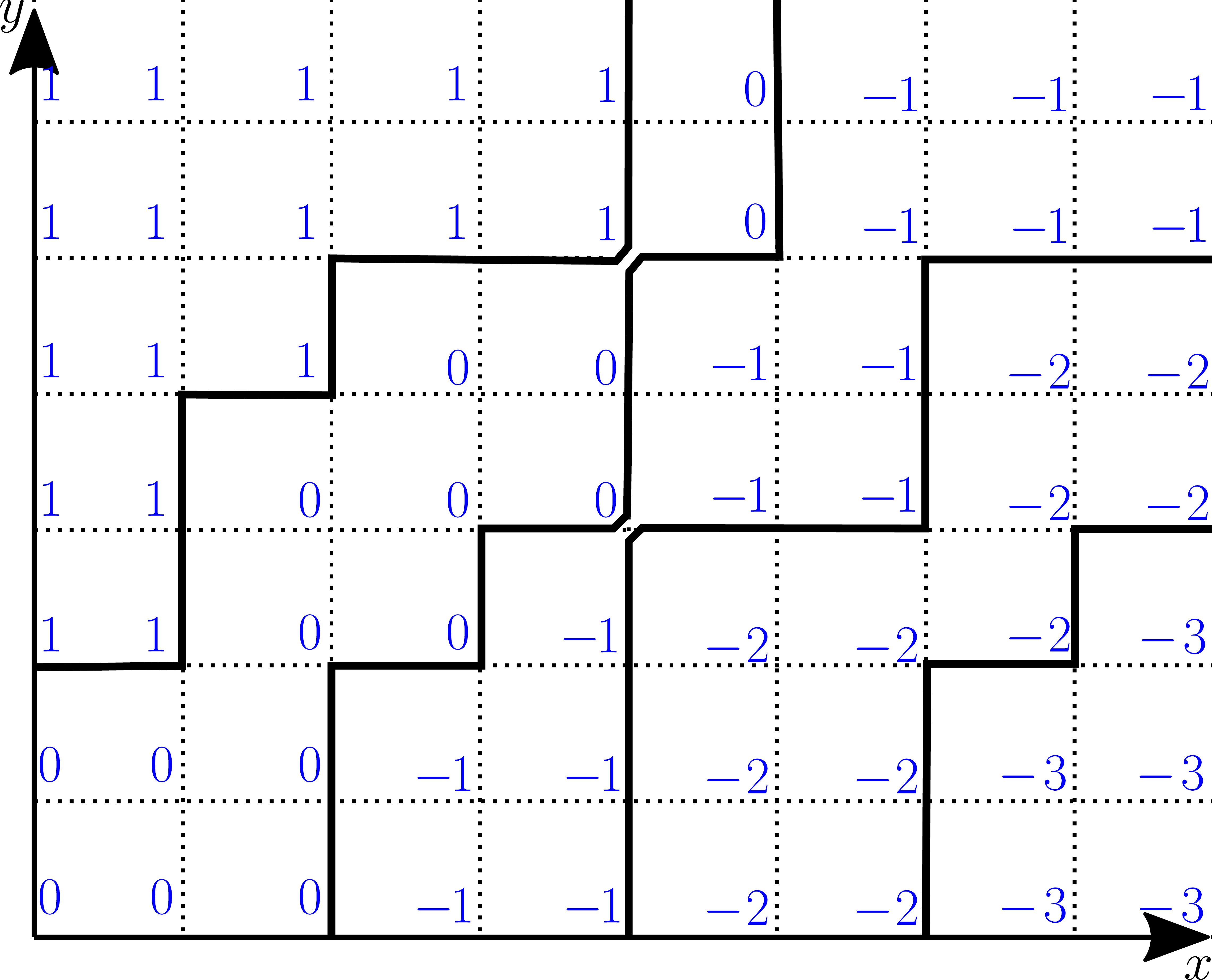}
\caption{The height function.}
\label{fig:6v}
\end{figure}

Recently, Borodin and Gorin \cite{borodin18} proposed a new scaling: with $ L\to\infty $ being the scaling parameter,
\begin{align}
	\label{eq:scaling}
	b_1=\exp\left(-\tfrac{\beta_1}{L}\right), \quad b_2=\exp\left(-\tfrac{\beta_2}{L}\right),
\end{align}
and scale space by $ L $: $ x,y \mapsto L^{-1}x, L^{-1}y $,
where $ \beta_1,\beta_2\in(0,\infty) $,  $ \beta_1\ne \beta_2 $, and fixed.
They showed that, under this scaling, the exponential height function converges to the \ac{TE}.
To state this result precisely, let us prepare some notation.
Set $ q:=b_1/b_2 $, $\q := e^{\beta_1-\beta_2} $ and consider
\begin{align}
	\label{eq:def-phi}
	\phi (x,y) := q^{H(x,y)}= {\q}^{\frac{1}{L} H(x,y)}=  e^{\frac{\beta_1-\beta_2}{L} H(x,y)}.
\end{align}
For given Lipschitz functions $ \chi,\psi:[0,\infty)\to\R $
with $ \chi(0)=\psi(0) $, it is known (\cite[Proposition~4.1,Theorem~4.4]{borodin18}) that  the \ac{TE}
\begin{align}
	\label{eq:te}
	\partial_{xy} \Phi+\beta_2\partial_x \Phi +\beta_1 \partial_y \Phi=0,
	\ x, y > 0,
	\quad
	\Phi(x,0)=\chi(x),\quad \Phi(0,y)=\psi(y).
\end{align}
admits a unique solution.
More explicitly, consider the Riemann function \cite[Eq.~(39)]{borodin18}
\begin{align}
	\label{eq:Rei}
 \Rie(x,y)=\frac{1}{2\pi \img} \oint_{\mathcal{C}} \exp\left[
 (\beta_1-\beta_2) \left(-x \frac{z}{z+ \beta_2} + y \frac{z}{z+\beta_1}\right)
 \right] \frac{(\beta_2-\beta_1)\, dz}{(z+\beta_1)(z+\beta_2)},
\end{align}
where the integration goes in positive direction and encircles $ -\beta_1 $, but not $ -\beta_2 $.
The solution $ \Phi $ of~\eqref{eq:te} is given by 
\begin{align}
\label{eq:Phi:mild}
 \Phi(x,y)
 =\psi(0) \Rie(x,y)+\int_0^y  \Rie(x,y-y')  \big(\psi'(y') + \beta_2 \psi(y') \big) dy'
+\int_0^x \Rie(x-x',y) \big(\chi'(x')+\beta_1\chi(x')\big) dx' .
\end{align}
\begin{defn}
\label{defn:scale}
For given $ f:\Z_{\geq 0}\to\R $ and $ g\in\Z^2_{\geq 0}\to\R $,
let $ f_L(x) := f(Lx) $ and $ g_L(x,y) := g(Lx,Ly) $ denote the corresponding scaled functions,
and linearly interpolate to be functions on $ \Rp $ and $ \Rp^2 $.
Linear interpolation from $ \frac{1}{L}\Z_{\geq 0} $ to $ \R $ is indeed unique.
To linearly interpolate from $ (\frac{1}{L}\Z_{\geq 0})^2 $ to $ \R^2 $, we fix a diagonal direction, say northeast-southwest,
and cut each square $ [i,i+1]\times[j,j+1] $, $ i,j\in\Z_{\geq 0} $, on the integer lattice into two triangles, diagonally along the prescribed direction.
This gives a triangulation on $ \Z_{\geq 0}^2 $ and hence on $ (\frac{1}{L}\Z_{\geq 0})^2 $, from which we construct a unique linear interpolation. 
\end{defn}

\begin{thm}[{\cite[Theorem~5.1]{borodin18}}]
\label{thm:lln}
Fix  Lipschitz functions $ \chi,\psi:[0,\infty)\to\R $,
and let $ \Phi $ be the unique solution of the telegraph equation~\eqref{eq:te} with boundary conditions $ \chi,\psi $,
given by~\eqref{eq:Phi:mild}.
If, as $ L\to\infty $, we have
\begin{align*}
	\sup_{x\in[0,a]}|\phi_L(\Cdot,x)-\chi(x)| \to 0 
	\quad
 	\text{and}
 	\quad
	\sup_{x\in[0,a]}|\phi_L(0,\Cdot) - \psi(x)| \to 0 
,
\end{align*}
fo reach $ a<\infty $, then, as $L\to \infty$,
\begin{align*}
	\sup_{(x,y)\in[0,a]^2} \big| \phi_L(x,y) - \Phi(x,y) \big|
	\longrightarrow_\text{P} 0,
	\quad
	\sup_{(x,y)\in[0,a]^2} \big| \tfrac1L H_L(x,y) - \h(x,y) \big|
	\longrightarrow_\text{P} 0,	
\end{align*}
for each $ a<\infty $, where $\h:=\log_\q \Phi$.
\end{thm}

As noted in \cite[Remark~5.3]{borodin18},
rewriting the equation~\eqref{eq:te} in terms of $ \h $-derivatives, and sending $ \q\to 0 $,
one obtains a nonlinear PDE that was observed in \cite{BCG2016,ResSri16b}
in the $ L\to\infty $ scaling limit but with $ b_1,b_2 $ fixed. 
Such a nonlinear PDE corresponds to inviscid/hyperbolic scaling limit in the context of hydrodynamic limits.
This is in contrast with the aforementioned \ac{SPDE}-limit results,
where the underlying hydrodynamic limits sit in the viscous/hyperbolic regime.
Given such an intriguing feature,
\cite{borodin18} further investigated the random fluctuations of $ \phi $ and  $ H $ around their respective means.
Our work here follows this study of random fluctuations.

Let $ u(x,y) := \phi(x,y)-\Ex[\phi(x,y)] $.
Let $ U(x,y) $, denote a centered Gaussian field on $ \Rp $, with covariance 
\begin{align*} 
	\Ex [U(x,y)U(x',y')]
	=
	\int_{0}^{x\wedge x'} \int_{0}^{y\wedge y'}
	\Rie(x-\bar x, y-\bar y) \ \Rie(x'-\bar x,y'-\bar y)
	\ 
	\D(\bar x, \bar y) \ d\bar x d\bar y,
\end{align*}
where
\begin{align}
\label{eq:D}
	\D(x,y) :=
	 (\beta_1+\beta_2) \partial_{ x} \Phi \cdot \partial_{ y}\Phi
		+ \beta_2 (\beta_2 -\beta_1) \partial_{ x} \Phi \cdot \Phi
		- \beta_1 (\beta_2 -\beta_1)  \Phi\cdot\partial_{ y} \Phi.
\end{align}
The following is our main result.

\begin{thm} 
\label{thm:main}
Under the same assumptions as in Theorem~\ref{thm:lln}, as $ L\to\infty $,
\begin{align*}
	\sqrt{L} u_L \to U
	\text{ in finite dimensional distributions}.
\end{align*}
\end{thm}
\begin{cor} 
\label{cor:main}
Under the same assumptions as in Theorem~\ref{thm:lln},
\begin{align*}
	\tfrac{1}{\sqrt{L}}\big( H(Lx,Ly) - \Ex[H(Lx,Ly)] \big) 
	\to \tilde U(x,y)
	:=
	\frac{U(x,y)}{(\beta_1-\beta_2) \Phi(x,y)}
	\text{ in finite dimensional distributions}.
\end{align*}
\end{cor}
\begin{rmk}
It is readily checked that
\begin{align*}
	U(x,y)
	 \stackrel{\text{law}}{=} 
	\int_{0}^{x} \int_{0}^{y} \Rie(x-x',y-y')\noise(x',y') \sqrt{D(x',y')} dx'dy',
\end{align*}
where $ \noise(x,y) $ denotes the Gaussian  white noise on $ \Rp^2 $.
Given such stochastic integral representation, we can also view $ U $
as the solution of the \ac{STE} with zero boundary condition, i.e.,
\begin{align}
	\label{eq:ste}
	\partial_{xy} U+\beta_2\partial_x U+\beta_1 \partial_y U = \sqrt{D}\noise,
	\quad
	x, y \geq 0,
	\quad
	U(x,0)=U(0,y)=0.
\end{align}
Alternatively,
substitute $ U=(\beta_1-\beta_2) \Phi \tilde{U} $ in~\eqref{eq:ste} and using~\eqref{eq:te},
we have the equation for $ \tilde{U} $:
\begin{align} \label{e:tildeUeq}
  \partial_{xy}\tilde{U}+\beta_1 \partial_{y}\tilde{U}
  + \beta_2 \partial_{x}\tilde{U}
   +(\beta_1-\beta_2) ( \partial_{y}\tilde{U} \partial_{x}\h + \partial_{x}\tilde{U}  \partial_{y}\h)
  =  \xi \cdot \sqrt{
  (\beta_1+\beta_2)
 \partial_{x}\h \partial_{y}\h
   - \beta_2 \, \partial_{x}\h
   + \beta_1\, \partial_{y}\h} .
\end{align}
\end{rmk}

Corollary~\ref{cor:main} was conjectured in \cite[Conjecture 6.1]{borodin18},
based on observations through a four point relation, and (separately) through  a variational principle and contour integrals.
For the low density regime (see \cite[Section~7]{borodin18} for the precise meaning), 
the analog of Corollary~\ref{cor:main} was established in \cite[Theorem~7.1]{borodin18}.
The main step toward proving such Gaussian limits is to show convergence of the variance.
Referring to~\eqref{eq:D}, we see that the variance involves $ \Phi $ and its gradients:
one term is quadratic in gradients, and the other terms are linear in gradients. 
In the low density regime, the quadratic-gradient term vanishes in the limit $ L\to\infty $,
and, through integration by parts,
\cite{borodin18} reduced convergence of the linear-gradient terms to convergence of $ \phi $
(i.e., the law of large numbers result in Theorem~\ref{thm:lln}),
whereby showing the convergence of $ u $.

For the general case (i.e., non-low-density) considered here,
one needs to address the convergence of the quadratic-gradient term.
The main tool we use here is the discrete, integrated form~\cite[Eq~(85)]{borodin18} of the \ac{STE}.
From this equation we develop expressions of discrete gradients of $ \phi $.
These expressions permit calculations of moments of the terms in question,
and from this we obtain decorrelation through contracting the discrete analog of $ \xi $.

\subsection*{Acknowledgements}
The authors thank Ivan Corwin for helpful discussions.
This work is initiated in the conference \textit{Integrable Probability Boston 2018}, in May 14-18, 2018 at MIT,
which is supported by the NSF through DMS-1664531, DMS-1664617, DMS-1664619, DMS-1664650.
Hao Shen is partially supported by the NSF through DMS:1712684.
Li-Cheng Tsai is partially supported by the NSF through DMS-1712575 and the Simons Foundation through a Junior Fellowship. 

\section{Preliminary}
\label{sec:prelim}
In this section we prepare a few tools for subsequent analysis.
Recall from \cite[Eq.~(45)]{borodin18} the discrete Riemann function
\begin{align}		
	\label{eq_Discrete_R}
 \rie(x,y)=\frac{1}{2\pi \img} \oint_{-\frac{1}{b_2(1-b_1)}} \left(\frac{1+  b_1(1-b_1)z}{1+ b_2(1-b_1) z}
  \right)^{x} \left( \frac{1+b_2 (1-b_2)z}{1+b_1(1-b_2)z} \right)^{y}
 \frac{ (b_2-b_1)\, dz}{(1+ b_2(1-b_1) z)(1+ b_1(1-b_2) z)},
\end{align}
where the integration goes in positive direction and encircles $-\frac{1}{b_2(1-b_1)}$, but not $-\frac{1}{b_1(1-b_2)}$.
We will also be using the notation $ \rieI(x,y) := \rie(x,y) \ind_\set{x\geq0} \ind_\set{y\geq 0} $.
Recall from \cite[Eq~(85)]{borodin18} the following integrated representation of $ \phi $, 
\begin{align}
\begin{split}
	\label{eq:phi}
	\phi&(x,y)
	= 
	\phi(0,0) \rie(x,y)
	+
  	\sum_{y'\in\Z_{> 0}} \rieI(x,y-y') \big(\phi(0,y')-b_2 \phi(0,y'-1)\big)
\\
	&+
  	\sum_{x'\in\Z_{> 0}} \rieI(x-x',y) \big(\phi(x',0)-b_1 \phi(x'-1,0)\big)
  	+\sum_{x',y'\in\Z_{> 0}} \rieI(x-x',y-y') \dnoise(x',y').
\end{split}
\end{align}
Here, $ \dnoise(x,y) $ is a process on $ \Z^2_{> 0} $ that plays the role of $ \xi $ (spacetime white noise) in the discrete setting.
In particular, with $ \nabla_x f(x) := f(x+1)-f(x) $ denoting the forward discrete gradient acting on a designated variable $ x $, 
recall from \cite[Theorem~3.1]{borodin18} that
\begin{align}
	\label{eq:condE-xi}
	&\Ex\big[ \dnoise(x+1,y+1) \mid
 	H(u,v), u\le x \text{ or } v\le y \big]=0,
\\ 
\begin{split}
	\label{eq:condVar-xi}
   &\Ex \big[ {\dnoise} (x+1,y+1)^2 \mid H(u,v), u\le x \text{ or } v\le y  \big]=
   \big(b_2(1-b_1)+b_1(1-b_2)\big) \nabla_x \phi (x,y)\nabla_y \phi(x,y)\\
   &\quad\quad+ b_1(1-b_2)(1-q) \phi(x,y)\nabla_x \phi(x,y)
   - b_1(1-b_1)(1-q) \phi(x,y) \nabla_y \phi(x,y).
\end{split}
\end{align}
Set $ \barphi(x,y) := \Ex[\phi(x,y)] $.
Indeed, since, on the r.h.s.\ of~\eqref{eq:phi}, only the last term is random, we have
\begin{align}
	\label{eq:u}
	u(x,y)
	&:=
	\phi(x,y) - \Ex[\phi(x,y)]
	=
	\sum_{x',y'\in\Z_{> 0}} \rieI(x-x',y-y') \dnoise(x',y'),
\\	
	\label{eq:barphi}
	\begin{split}
		\barphi(x,y)
		&=
		\phi(0,0) \rie(x,y)
		+
	  	\sum_{y'\in\Z_{> 0}} \rieI(x,y-y') \big(\phi(0,y')-b_2 \phi(0,y'-1)\big)
	\\
		&
		\hphantom{=\phi(0,0) \rie(x,y)}
		+
	  	\sum_{x'\in\Z_{> 0}} \rieI(x-x',y) \big(\phi(x',0)-b_1 \phi(x'-1,0)\big).
  	\end{split}
\end{align}

Hereafter, we use $ c(a,b,\ldots)<\infty $ to denote a generic finite constant
that may change from line to line, but depends only on the designated variables $ a,b,\ldots $.
The parameter $ \beta_1\neq \beta_2 $ are considered fixed, so their dependence will be omitted.
\begin{lem} \label{lem:bound-R}
For any $k=(k_1,k_2)\in \Z_{\geq 0}^2$ we write $|k|=k_1+k_2$,
and $\partial^k=\partial_x^{k_1}\partial_y^{k_2}$
and $\nabla^k =\nabla_x^{k_1} \nabla_y^{k_2}$.
Given any $m\in \Z_{>0}$ and $ a<\infty $,
\begin{align}
	\label{eq:Riebdd}
	&\sum_{0\le |k| \le m} |\partial^k \Rie(x,y) |  \le c(a,m),
	\quad
	\forall (x,y) \in[0,a]^2,
\\
	\label{eq:riebdd}
	& \sum_{0\le |k| \le m} |L^k \nabla^k \rie(x,y) |  \le c(a,m),
	\quad
	\forall 
	(x,y) \in([0,aL]\cap\Z)^2,
\\
	\label{eq:rietoRie}
	&
	\lim_{L\to \infty}\sup_{(x,y)\in ([0,aL]\cap\Z)^2} \sum_{0\le |k| \le m}
	|\partial^k \Rie(\tfrac{x}{L},\tfrac{y}{L}) -L^k \nabla^k \rie(x,y)| =0.
\end{align}
\end{lem}
\begin{proof}
Consider the formula \eqref{eq:Rei} for $\Rie$,
and \emph{fix} a contour $ \mathcal{C} $ as described therein.
This is a closed curve of finite length, and along the contour $ z\in\mathcal C $ are bounded in absolute value, i.e., $ |z|\leq c $.
Each of the factors in the integrand is bounded over $[0,a]^2$ uniformly in $ z\in \mathcal C $.
Moreover, each $\partial_x$ brings down a factor
$- \frac{(\beta_1-\beta_2)  z}{z+ \beta_2}$
and each $\partial_y$ brings down a factor
$\frac{(\beta_1-\beta_2)  z}{z+ \beta_1}$;
these are all bounded uniformly in $z\in \mathcal C$.
From these discussions we conclude~\eqref{eq:Riebdd}.


Noting that~\eqref{eq:riebdd} follows from~\eqref{eq:Riebdd} and~\eqref{eq:rietoRie},
we now move on to proving~\eqref{eq:rietoRie}.
Apply changes of variables to \eqref{eq_Discrete_R}:
$z=L\tilde z/(\beta_1\beta_2)$, $x=L\tilde x$, $y=L \tilde y$.
Then
\begin{align}
\label{eq:riechange}
\begin{split}
 \rie(\tilde{x},\tilde{y})
 =
 \frac{1}{2\pi \img} 
 \oint_{-\frac{\beta_1\beta_2}{L b_2(1-b_1)}}
  &\left(\frac{\beta_1\beta_2+  L b_1(1-b_1) \tilde z}{\beta_1\beta_2+ L b_2(1-b_1) \tilde z}
  \right)^{L\tilde x} 
\left( \frac{\beta_1\beta_2+L b_2 (1-b_2)\tilde z}{\beta_1\beta_2+L b_1(1-b_2)\tilde z} \right)^{L\tilde y}
\\
  &
 \cdot\,\frac{ L\beta_1\beta_2(b_2-b_1)\, d\tilde z}{(\beta_1\beta_2+ L b_2(1-b_1) \tilde z)(\beta_1\beta_2+ L b_1(1-b_2) \tilde z)},
\end{split}
\end{align}
where the integration goes in positive
direction and encircles $-\frac{\beta_1\beta_2}{L b_2(1-b_1)}$ 
but not $-\frac{\beta_1\beta_2}{Lb_1(1-b_2)}$.
Indeed, as $ L\to\infty $,
$ -\frac{\beta_1\beta_2}{L b_2(1-b_1)} \to -\beta_2 $ and $ -\frac{\beta_1\beta_2}{Lb_1(1-b_2)}\to-\beta_1 $.
This being the case, we fix a contour $\mathcal C'$ (\emph{independently} of $ L $)
that goes in the positive direction encircling $ -\beta_2 $ but not $ -\beta_1 $.
It is readily checked that, \emph{uniformly} over $ \tilde{z}\in\mathcal{C}' $ and $ \tilde{x},\tilde{y}\in [0,a] $, as $ L\to\infty $,
\begin{align*}
	\left(\frac{\beta_1\beta_2+  L b_1(1-b_1) \tilde z}{\beta_1\beta_2+ L b_2(1-b_1) \tilde z} \right)^{L\tilde x}
	&\longrightarrow
	\exp\Big( (\beta_1-\beta_2) \Big(-\tilde{x} \frac{\tilde{z}}{\tilde{z}+ \beta_2} \Big) \Big),
\\
	\left( \frac{\beta_1\beta_2+L b_2 (1-b_2)\tilde z}{\beta_1\beta_2+L b_1(1-b_2)\tilde z} \right)^{L\tilde y}
	&\longrightarrow
	\exp\Big( (\beta_1-\beta_2) \Big(-\tilde{y} \frac{\tilde{z}}{\tilde{z}+ \beta_1} \Big) \Big),
\\
	\frac{ L\beta_1\beta_2(b_2-b_1)}{(\beta_1\beta_2+ L b_2(1-b_1))(\beta_1\beta_2+ L b_1(1-b_2) \tilde z)}
	&\longrightarrow
	-\frac{(\beta_2-\beta_1)}{(\tilde{z}+\beta_1)(\tilde{z}+\beta_2)}.
\end{align*}
Using this in~\eqref{eq:riechange} gives, \emph{uniformly} over $ \tilde{x},\tilde{y}\in [0,a] $ as $ L\to\infty $,
\begin{align}
\label{eq:riechange:}
\begin{split}
 \rie(\tilde{x},\tilde{y})
 \longrightarrow
 -\frac{1}{2\pi \img} 
 \oint_{\mathcal{C}'}
 \exp\Big( (\beta_1-\beta_2) \Big(-\tilde{x} \frac{\tilde{z}}{\tilde{z}+ \beta_2}+\tilde{y} \frac{\tilde{z}}{\tilde{z}+ \beta_1} \Big) \Big)
\frac{(\beta_2-\beta_1)\,d\tilde{z}}{(\tilde{z}+\beta_1)(\tilde{z}+\beta_2)}.
\end{split}
\end{align}
Note that, compared to~\eqref{eq:Rei},
the r.h.s.\ of~\eqref{eq:riechange:} has a different contour $ \mathcal{C}' $, and an outstanding negative sign.
However, as noted in \cite{borodin18} (see comments after Equation~(39) therein),
the integrand in \eqref{eq:Rei} and \eqref{eq:riechange:} has no pole at $ |\tilde{z}|=\infty $,
so the contour $ \mathcal{C}' $ can be deformed to $ -\mathcal{C} $ (the orientation changes after deformation),
matching the r.h.s.\ of~\eqref{eq:riechange:} to \eqref{eq:Rei}.
This proves \eqref{eq:rietoRie} for $ |k|=0 $.
%
As for $ |k|>0 $,
note that each $L\nabla_x$ applied to \eqref{eq_Discrete_R} brings a factor $ \frac{L(b_1-b_2)(1-b_1)z}{1+b_2(1-b_1)z} $, 
and each $ L\nabla_y $ brings a factor $ \frac{L(b_2-b_1)(1-b_2)z}{1+b_1(1-b_2)z} $. 
These factors converges uniformly over $ \tilde{z}\in\mathcal{C}' $ to
$ \frac{(\beta_1-\beta_2)  \tilde z}{\tilde z+ \beta_2} $ and $ \frac{(\beta_1-\beta_2)  \tilde z}{\tilde z+ \beta_1 } $, respectively.
Hence~\eqref{eq:rietoRie} follows.
\end{proof}

\begin{lem} \label{lem:apribd}
Given $ a<\infty $, we have, for all $(x,y)\in ([0,aL]\cap \Z)^2$,
\begin{align}
	\label{eq:phi:apribd}
	|\phi(x,& y)| + |L\nabla_x\phi(x,y)| + |L\nabla_y\phi(x,y)|
  	\le c(a) 
\\
	\label{eq:barphi:apribd}
    |\barphi(x,& y)| + |L\nabla_x\barphi(x,y)| + |L\nabla_y\barphi(x,y)|
	\le c(a),
\\
	\label{eq:barphi:cnvg}
	\sup_{(x,y)\in[0,a]^2}
	&| \barphi_L(x,y)-\Phi(x,y) | \longrightarrow 0,
	\quad
	\text{as } L\to\infty.
\end{align}
\end{lem}
\begin{proof}
Given Lemma~\ref{lem:bound-R}, \eqref{eq:barphi:apribd}--\eqref{eq:barphi:cnvg} are readily verified from~\eqref{eq:barphi}.
As for~\eqref{eq:phi:apribd}, recall that $ \q:=e^{\beta_1-\beta_2} $ is fixed.
Indeed, since $ H(0,0):=0 $ by definition,
 and $ \nabla_xH(x,y)\in\{0,1\} $ and $ \nabla_yH(x,y)\in\{0,1\} $, we have
$	\phi(x,y) := \q^{\frac{1}{L}H(x,y)} \leq \q^{\frac{2La}{L}} =c(a),$
and
$	|L\nabla_\beta\phi(x,y)| = |L(\q^{\frac{1}{L}\nabla_aH(x,y)}-1)\phi(x,y)| \leq c(a),$
for	
$ \beta=x,y $.
%
%
%
\end{proof}

\begin{lem}  \label{lem:cond-mom-xi}
For any $ k\in\Z_{> 0} $, 
\begin{align}	
	\label{eq:dnoise:mom}
 	\Ex\big[ \, |\dnoise(x+1,y+1)|^k \, & \big| \, H(x',y'), x'\le x \text{ or } y'\le y \big] \le c(k) L^{-k-1},
\\
	\label{eq:dnoise:uniform}
	&|\dnoise(x+1,y+1)| \leq cL^{-1},
\end{align}
for all $(x,y)\in \Z_{\geq 0}^2 $, $ L\geq 1 $.
\end{lem}
\begin{proof}
First, conditioning $ H(x',y'), x'\le x \text{ or } y'\le y $ for $ \dnoise(x+1,y+1) $
amounts to conditioning on incoming line configuration into the site $ (x+1,y+1) $.
There are four cases pertaining to such conditions, 
and in each case  $\dnoise(x+1,y+1) $ is computed in \cite[Proof of Theorem~3.1]{borodin18},
using the `four point relation' derived therein.
We record the results of their computation here, and examine 
the asymptotics in $L$ of the values of $\dnoise$ and their probabilities in each case.
In the following vertices of type $ I $--$ VI $ refers to those depicted in Figure~\ref{tbl:S6VWeights}.

\begin{enumerate}
\item
No line enters into the vertex $(x+1,y+1)$ from below or from the left:
In this case the vertex is of type $I$, whereby
$
	\dnoise(x+1,y+1)=0.
$
\item
Two lines enter into the vertex $(x+1,y+1)$, one from below and one from the left:
In this case the vertex is of type $II$, whereby
$
	\dnoise(x+1,y+1)=0.
$
\item
One line enters into the vertex $(x+1,y+1)$ from below, but no line enters from the left: 
In this case the vertex is of type $IV$ with probability $b_2$ and
 \[
 |\dnoise(x+1,y+1)|= |q^h (q^{-1}-b_1)(1-q)| \le c L^{-2};
 \]
or of type $VI$ with probability $1-b_2 \le c L^{-1}$ and
 \[
 |\dnoise(x+1,y+1)|= |q^h b_1 (q-1)| \le c L^{-1} .
   \]
\item
One line enters into the vertex $(x+1,y+1)$ from the left, but no line enters from below: 
In this case the vertex is of type $III$ with probability $b_1$ and
\[
|\dnoise(x+1,y+1)|=|q^h(1-b_1)(q-1)|\le cL^{-2};
\]
or of type $V$ with probability $1-b_1 \le c L^{-1}$ and
\[
| \dnoise(x+1,y+1)|=|q^h b_1(1-q)| \le cL^{-1}.
\]
\end{enumerate}
The conditional moments bound \eqref{eq:dnoise:mom}
and the uniform bound \eqref{eq:dnoise:uniform} readily follow from the preceding discussion.
\end{proof}

\section{Proof of Theorem~\ref{thm:main} and Corollary~\ref{cor:main}}
\label{sec:pfmain}
\subsection{Proof of Theorem~\ref{thm:main}}
Write $ \Rightarrow $ for convergence in distribution.
Hereafter throughout the article, we fix $ (\x_1,\y_1),\ldots, (\x_n,\y_n) \in \Rp^2 $.
Our goal is to prove 
$
	(\sqrt{L}u(L\x_i,L\y_i))_{i=1}^n \Rightarrow (U(\x_i,\y_i))_{i=1}^n.
$
To simplify notation, we work under the consent that whenever the arguments of $ u $ are not integers,
they are being taken integer parts, e.g., $ u(L\x_i,L\y_i):= u(\lfloor L\x_i\rfloor,\lfloor L\y_i\rfloor) $.
Similar convention is adopted without explicitly stated for processes over integers.

Given the expression~\eqref{eq:u} of $ u $,
we proceed via martingale \ac{CLT}, (as in \cite{borodin18} for the low density regime).
To this end we linearly order points on $ \Z^2_{>0} $ as
\begin{align}
	\label{eq:linord}
	(x(1),y(1)):=(1,1),\quad
	(x(2),y(2)):=(2,1),\quad
	(x(3),y(3)):=(1,2),\quad
	(x(4),y(4)):=(3,1),\quad \cdots
\end{align}
%
%
Consider the discrete time process $ M(t)\in\R^d $, $ t=1,2,\ldots $,
\begin{align}
	\label{eq:M}
	M(t):=(M_i(t))_{i=1}^n,
	\quad
	M_i (t)
	:= 
	\sum_{s=1}^t \sqrt{L} \rieI(L \x_i - x(s),L\y_i - y(s)) \, \dnoise(x(s),y(s)).
\end{align}
It follows from~\eqref{eq:condE-xi} that $ M(t) $ is a martingale.
Recall that, by definition, $ \rieI(x,y) $ carries indicator functions forcing $ x,y\geq 0 $.
Hence, for some large enough $ c_*<\infty $,
\begin{align}
	\label{eq:c*}
	M(c_*L^2)=M(c_*L^2+1)=\ldots=M(\infty) 
	= (\sqrt{L} u(L\x_i,L \y_i))_{i=1}^n.
\end{align}
Let $\filtM(t) := \sigma(M(1),\ldots,M(t)) $ denote the canonical filtration of $ M(t) $,
and recall that cross variance of $ M $ is defined as
\begin{align}
	\label{eq:qv:}
	\langle M_i,M_j\rangle (t)
	&:= \sum_{s=1}^t \Ex \big[ (M_i(s)-M_i(s-1))(M_j(s)-M_j(s-1)) \big| \filtM(s-1)\big].
\end{align}
Put $ \RieI(x,y) := \Rie(x,y)\ind_\set{x\geq 0}\ind_\set{y\geq 0} $,
and recall the definition of $ \D(x,y) $ from~\eqref{eq:D}. We set
\begin{align}
	\label{eq:QV}
	\QV_{ij}
	:=
	\int_{\Rp^2} \RieI_{ij}(x,y)\, D(x,y) dxdy,
	\quad
	\RieI_{ij}(x,y) := \RieI(\x_i-x,\y_i-y)\RieI(\x_j-x,\y_j-y).
\end{align}

The martingale \ac{CLT} from \cite{hall2014} applied to $ M(t) $ gives
\begin{thm}[{\cite[Corollary~3.1]{hall2014}}]
\label{thm:mgCLT}
If, for any $ i,j=1,\ldots,n $ and $ \e>0 $,
\begin{align}
	\tag{\text{Lind}}
	\label{eq:lindCnd}
	& 
	\sum_{s=1}^{L^2c_*} \Ex\big[ (M_i(s)-M_i(s-1))^2 \ind_\set{|M_i(s)-M_i(s-1)|>\e} \big] 
	\longrightarrow 0,
\\
	\tag{\text{QV}}
	\label{eq:qvCnd}
	&\langle M_i,M_j\rangle (L^2c_*) \Longrightarrow_\text{P}
	\QV_{ij},
\end{align}
then
\begin{align*}
	M(c_*L^2) \Longrightarrow (U(\x_i,\y_i))_{i=1}^n.
\end{align*}
\end{thm}

\begin{rmk}
Note that, even though~\cite[Corollary~3.1]{hall2014} is stated for $ \R $-valued martingale,
generalization to $ \R^n $-value is standard, by projection $ M(t)\in\R^n $ onto arbitrarily fixed $ v\in\R^n $.
\end{rmk}

Given Theorem~\ref{thm:mgCLT}, it suffices to check the conditions~\eqref{eq:lindCnd}--\eqref{eq:qvCnd}.
The former follows at once from the fact that $ |\dnoise(x,y)| \leq cL^{-1} $ (from Lemma~\ref{lem:cond-mom-xi}),
which makes the indicator functions in~\eqref{eq:lindCnd} zero for all large enough $ L $. 
We hence devote the rest of the article to proving~\eqref{eq:qvCnd}. 
From \eqref{eq:condVar-xi} we calculate the cross variance (defined in~\eqref{eq:qv:}) as
\begin{align}
	\langle M_i,M_j\rangle (c_*L^2)
	&= L^{-2}\sum_{s=1}^{c_*L^2}
	 \rieI_{ij}(x(s),y(s)) \,  \Dd(x(s),y(s);\phi),
 	\label{eq:qv}
\end{align}
where
\begin{align}
	\label{eq:rieIij}
	\rieI_{ij}(x,y)
	&:=
	\rieI(L \x_i - x,L\y_i - y)  \rieI(L \x_j - x,L\y_j - y) 
\\
	\label{eq:Dd}
	\Dd(x,y;\phi)  
	&:=
	\gamma_{xy} \cdot
	L \nabla_x \phi (x,y) \cdot L \nabla_y \phi(x,y)
	+\gamma_{x}\cdot \phi(x,y) \cdot L\nabla_x \phi(x,y)
	+ \gamma_{y}\cdot \phi(x,y) \cdot L\nabla_y \phi(x,y),
\\
	\label{eq:gamma}
	\gamma_{xy}
	&:= L(b_2(1-b_1)+b_1(1-b_2)),
	\
	\gamma_{x}:=
 	L^2(b_1(1-b_2)(1-q)),
 	\
 	\gamma_{y}:=
   - L^2(b_1(1-b_1)(1-q)).
\end{align}
Recall that $ \barphi(x,y) := \Ex[\phi(x,y)] $.
Compare~\eqref{eq:QV} and \eqref{eq:qv}--\eqref{eq:Dd}.
The main step toward proving~\eqref{eq:qvCnd} is to show that, in~\eqref{eq:Dd},
we can approximate $ \phi $, $ L\nabla_x\phi $, and $ L\nabla_y\phi $ 
by their continuum counterparts $ \Phi $, $ \partial_x\Phi $, and $ \partial_y\Phi $,
in a suitable sense under the limit $ L\to\infty $.
With this in mind, we decompose $ \langle M_i, M_j \rangle(L^2c_*)-\QV_{ij} = S_1+S_2 $, where
\begin{align}
	\label{eq:S1}
	S_1 
	&:=
	L^{-2}\sum_{s=1}^{c_*L^2}
	\rieI_{ij}(x(s),y(s)) \,  \Dd(x(s),y(s);\barphi)
	-
	\QV_{ij},
\\
	\label{eq:S2}
	S_2
	&:=
	L^{-2}\sum_{s=1}^{c_*L^2}
	\rieI_{ij}(x(s),y(s)) \, \big( \Dd(x(s),y(s);\phi) - \Dd(x(s),y(s);\barphi) \big).
\end{align}
Here, $ S_2 $ records the difference of replacing $ \phi $, $ L\nabla_x\phi $, and $ L\nabla_y\phi $
with their respective expectations $ \barphi $, $ L\nabla_x\barphi $, and $ L\nabla_y\barphi $;
while $ S_1 $ accounts for the difference between $ \barphi $, $ L\nabla_x\barphi $, and $ L\nabla_y\barphi $
with their corresponding terms in continuum $ \Phi $, $ \partial_x\Phi $, and $ \partial_y\Phi $.
In particular, note that $ S_1 $ is deterministic.
We will show separately that $ S_1\to 0 $ and $ S_2\to_\text{P} 0 $:
\begin{prop}
\label{prop:qv}
For fixed $ i,j\in\{1,\ldots,n\} $, with $ S_1 $ and $ S_2 $ defined in~\eqref{eq:S1}--\eqref{eq:S2},
we have, as $ L\to\infty $,\\
\begin{enumerate*}[label=(\alph*)]
\item \label{prop:qv:S1} $ S_1\to 0 $; 
and \item \label{prop:qv:S2} $ S_2\to_\text{P} 0 $.
\end{enumerate*}
\end{prop}
\noindent
Proposition~\ref{prop:qv} verifies the condition~\eqref{eq:qvCnd} and hence completes the proof of Theorem~\ref{thm:main}.
The proof of Proposition~\ref{prop:qv} is carried out in Sections~\ref{sec:S1}--\ref{sec:S2}.

\subsection{Proof of Corollary~\ref{cor:main}}
Fix $ a<\infty $, throughout this proof we assume $ x,y\in[0,aL]\cap\Z $ and write $ c=c(a) $ to simplify notation.
The first step is to express $ H $ in term of $ \phi $.
To this end, write 
\begin{align}
	\label{eq:Hlog}
	H(x,y) = L\log_\q\big(\phi(x,y)\big) = L\log_\q\big( \barphi(x,y) +  u(x,y) \big).
\end{align}
Recall that $ \phi= \q^{\frac1L H} $,
and that $ H(0,0):=0 $ and $ H $ is $ 1 $-Lipschitz (from the definition of height function).
Hence 
\begin{align}
	\label{eq:bdd0inf}
	\tfrac{1}{c} \leq \phi(x,y) \leq c,
	\quad
	\tfrac{1}{c} \leq \barphi(x,y) \leq c,
	\quad
 	x,y \in[0,aL]\cap\Z.
\end{align}
In~\eqref{eq:Hlog}, Taylor expand the function $ \log_\q(\barphi+u) $ in $ u $ to the first order, with the aid preceding bounds, we have
\begin{align}
	\label{eq:Hexpand}
	H(x,y) = L \log_\q(\barphi(x,y)) + \tfrac{L}{\barphi(x,y)\log\q} u(x,y) + LR(x,y), 
\end{align}
for some remainder $ R $ such that 
\begin{align}
	\label{eq:Rbd}
	|R(x,y)| \leq cu^2(x,y).
\end{align}

Take expectation in~\eqref{eq:Hexpand}, subtract the result from~\eqref{eq:Hexpand}.
With $ \Ex[u(x,y)]=0 $, we have
\begin{align*}
	H(x,y) - \Ex[H(x,y)]
	= 
	\tfrac{L}{\barphi(x,y)\log\q} u(x,y) + L\big(R(x,y) - \Ex[R(x,y)]\big), 
\end{align*}
Recall the scaling convention from Definition~\ref{defn:scale}.
Divide both sides by $ \sqrt{L} $, with $ \log\q = \beta_1-\beta_2 $, we have
\begin{align}
	\label{eq:corH}
	\tfrac{1}{\sqrt{L}} \big( H_L - \Ex[H_L] \big)
	= 
	\tfrac{1}{\barphi_L(\beta_1-\beta_2)} \sqrt{L} u_L + \sqrt{L}\big(R_L - \Ex[R_L]\big), 
\end{align}

From Theorem~\ref{thm:main} we already have $ \sqrt{L} u_L \to U $ in finite dimensional distributions.
This together with~\eqref{eq:barphi:cnvg} and~\eqref{eq:bdd0inf} gives
$ \frac{1}{\barphi_L(\beta_1-\beta_2)} \sqrt{L} u_L \to \tilde{U} $ in finite dimensional distributions.
To control the last term in~\eqref{eq:corH}, we calculate the second moment of $ u(x,y) $ from~\eqref{eq:u}.
By \eqref{eq:condE-xi}, the discrete noise $ \dnoise(x,y) $, $ x,y\in\Z_{> 0} $ are uncorrelated, so
$
	\Ex[ u(x,y)^2 ]
	=
	\sum_{x',y'\in\Z_{> 0}} \rieI(x-x',y-y')^2 \Ex[ \dnoise(x',y')^2 ].
$
Further using the bounds on $ \rie $ from Lemma~\ref{lem:bound-R} and the bound on $ \Ex[ \dnoise(x',y')^2 ] $
from Lemma~\ref{lem:cond-mom-xi}, we conclude $ \Ex[u(x,y)^2] \leq cL^{-1} $.
Combining this with \eqref{eq:Rbd} gives $ \Ex|R(x,y)| \leq cL^{-1} $.
From this, we see that the last term in~\eqref{eq:corH} converges to zero
in finite dimensional distributions. 
This completes the proof.

\section{Proof of Proposition~\ref{prop:qv}\ref{prop:qv:S1}}
\label{sec:S1}
Recall that $ (\x_1,\y_1),\ldots,(\x_n,\y_n)\in\Rp^2 $ are points fixed previously.
Hereafter, we fix further $ i,j\in\{1,\ldots,n\} $.
Recall from Lemma~\ref{lem:apribd} 
that $ \barphi_L(x,y) = \barphi(Lx,Ly) $ converges uniformly to $ \Phi $.
On the other hand, from the integrated representation~\eqref{eq:barphi},
it is not hard to check that $ L \nabla_x \phi(Lx,Ly) \not\to \partial_x\Phi $
and $ L \nabla_y \phi(Lx,Ly) \not\to \partial_y\Phi $ in general.
That is, derivatives of $ \barphi $ do \emph{not} converge pointwisely.
Given that the quantities $ \D $ and $ \Dd $ (defined in~\eqref{eq:D} and \eqref{eq:Dd}) involves gradients, 
in order to show $ S_1\to\infty $, 
one needs to exploit the sum over $ s $ in~\eqref{eq:S1}, as well as the integral over $ x,y $ in~\eqref{eq:QV}.
The sum and integral smear out the possibly fluctuating derivatives.
In the following two lemmas we expose the aforementioned smearing effect
via integration-by-parts and summation-by-parts formulas.

Let $ \Lip(\Rp) $ and $ \Lip(\Rp^2) $ denote the spaces of functions 
that are uniformly Lipschitz respectively over compact subsets of $ \Rp $ and $ \Rp^2 $.
Following the preceding discussion, instead of Lipschitz norms,
we equip $ \Lip(\Rp) $ and $ \Lip(\Rp^2) $
with the topology of uniform convergence over compact subsets.
Recall that $ \RieI_{ij}(x,y) := \RieI(\x_i-x, \y_i-y) \RieI(\x_j-x,\y_j-y) $.
For $ \alpha=x,y $, consider the map
\begin{align}
	\label{eq:U}
	\V_\alpha : \Lip(\Rp^2) \longrightarrow \R,
	\quad
	\V_\alpha(\Phi)
	&:=
	\int_{\Rp^2}
	\RieI_{ij}(x,y) \ \partial_{\alpha} \Phi (x,y) \cdot\Phi(x,y) \ dx dy.
\end{align}
For given $ \chi,\psi\in\Lip(\Rp) $, let $ \Phi=\Phi(\chi,\psi) $ defined through~\eqref{eq:Phi:mild}.
Consider the following map $ \V_{xy}: \Lip(\Rp)\times\Lip(\Rp)\to\R $:
\begin{align} \label{e:def-Vxy}
	\V_{xy}(\chi,\psi)
	&:=
	\int_{\Rp^2}
	\RieI_{ij}(x,y)
	\,
	\partial_{x} \Phi(x,y)
	\cdot
	\partial_{y} \Phi(x,y)
	\ dxdy,
	\quad
	\Phi=\Phi(\chi,\psi) \text{ via }\eqref{eq:Phi:mild}.
\end{align}

\begin{lem}
\label{lem:ibp} 
For $ \alpha,\beta\in\{x,y\} $,
the maps $ \V_\alpha : \Lip(\Rp^2) \to \R $, $ \V_{xy} : \Lip(\Rp)^2 \to \R $ are continuous
(under uniform topology, as declared previously).
%
%
\end{lem}

\begin{proof}
We begin with $ \V_\alpha $.
Take $ \alpha=x $ to simplify notation. The case $ \alpha=y $ follows exactly the same.
To simplify notation,
set $ \barx:=\x_i\wedge\x_j $ and $ \bary:=\y_i\wedge\y_j $
and $ \Rie_{ij}(x,y) := \Rie(\x_i-x, \y_i-y) \Rie(\x_j-x,\y_j-y) $.
In~\eqref{eq:U}, writing $ \partial_x\Phi\cdot\Phi =\frac12\partial_x(\Phi^2) $, we have
\begin{align*}
	\V_x(\Phi)
	=
	\frac12\int_{0}^{\barx}\int_{0}^{\bary}
	\Rie_{ij}(x,y)
	\,
	\partial_x(\Phi^2(x,y))
	\, dx dy.
\end{align*}
Integration by parts in $ x $ gives
\begin{align}
\label{eq:VVx}
	\V_x(\Phi)
	=
	\frac12 \int_{0}^{\bary}
	\Rie_{ij}(x,y) \Phi^2(x,y) \Big|_{x=0}^{x=\barx} dy
	-
	\frac12\int_{0}^{\barx}\int_{0}^{\bary}
	\partial_x \Rie_{ij}(x,y) \cdot
	\Phi^2(x,y)dxdy.
\end{align}
Given that $ \Rie_{ij} $ is smooth (from Lemma~\ref{lem:bound-R}),
from~\eqref{eq:VVx} it is clear that $ \V_x $ is continuous in $\Phi$.

Turning to $ \V_{xy} $,
take $ x $-derivative in \eqref{eq:Phi:mild} to get
$ \partial_x\Phi(x,y) =  \Rie(0,y) \chi'(x) + G (x,y) $,
where
\begin{align}
	\notag
	G (x,y)
	:=&\psi(0) \partial_x \Rie(x,y) + \beta_1\Rie(0,y)\chi(x)
\\
	\label{eq:Gx}
	& +\int_0^y \partial_x \Rie(x,y-y')  \big(\psi'(y') + \beta_2 \psi(y') \big) dy'  
	+ \int_0^x \partial_x \Rie(x-x',y) \big(\chi'(x')+\beta_1\chi(x')\big) dx'.
\end{align}
Note that $ G $ involves the derivatives $ \chi' $ and $ \psi' $.
We integrate by parts to separate the dependence on $ \chi' $ and $ \psi' $ from the dependence on $ \chi $ and $ \psi $.
To state this precisely, consider the set $ \Generic $ that consists of
finite linear combinations of the following expressions
\begin{align}
	\label{eq:generic}
	\partial^{k_1} \Rie(x-x_1,y) \chi(x_1)
	\quad
	\partial^{k_2} \Rie(x,y-y_2) \psi(y_2),		
	\quad
	\int_0^x \partial^{k_3} \Rie(x-x',y) \chi(x') dx',
	\quad
	\int_0^y \partial^{k_4} \Rie(x,y-y') \psi(y') dy',
\end{align}
where $ k_i =(k_i,k'_i)\in\Z_{\geq 0}^2 $ are multi-indices (defined in Lemma~\ref{lem:bound-R}) with $ |k_i| \leq 3 $,
$ x_1\in\{0,x\} $, and $ y_2\in\{0,y\} $.
That is,
\begin{align}
	\label{eq:Generic}
	\Generic 
	:= 
	\Big\{ 
		\sum \alpha_\text{term} \cdot \big(\text{term in~\eqref{eq:generic}}\big) 
	\Big\}.
\end{align}
In~\eqref{eq:Gx}, integrating by parts in $ y' $ and in $ x' $ respectively for the first and second integrals, we have
\begin{align}
	\label{eq:Phix}
	\partial_x \Phi(x,y)
	=
	\Rie(0,y)\chi'(x) + \partial_x\Rie(x,0)\psi(y) + \generic_y,
\end{align}
for some $ \generic_y $ such that $ \generic_y, \partial_y \generic_y \in \Generic $.
A similarly calculation applied to $ \partial_y \Phi $ gives
\begin{align}
	\label{eq:Phiy}
	\partial_y \Phi(x,y)
	=
	\Rie(x,0)\psi'(y) + \partial_y\Rie(0,y)\chi(x) + \generic_x,
\end{align}
for some $ \generic_x $ such that $ \generic_x, \partial_x \generic_x \in \Generic $.

Inserting~\eqref{eq:Phix}--\eqref{eq:Phiy} into~\eqref{e:def-Vxy} gives
\begin{subequations}
\begin{align}
	\label{eq:Vxy1}
	\V_{xy}(\chi,\psi) 
	=& 
	\I\big( \Rie(0,y)\chi'(x)\cdot\Rie(x,0)\psi'(y) \big)
\\
	\label{eq:Vxy2}
	&
	+\I\big( \Rie(0,y)\chi'(x)\cdot\partial_y\Rie(0,y)\chi(x) \big) + \I\big( \partial_x\Rie(x,0)\psi(y)\cdot\Rie(x,0)\psi'(y)  \big)	
\\
	\label{eq:Vxy3}
	&
	+\I\big( \Rie(0,y)\chi'(x)\cdot\generic_x \big) + \I\big( \generic_y\cdot\Rie(x,0)\psi'(y)  \big)	
\\
	\label{eq:Vxy4}
	&
	+\I\big( (\partial_x\Rie(x,0)\psi(y) + \generic_y)\cdot(\partial_y\Rie(0,y)\chi(x) + \generic_x)  \big),
\end{align}
\end{subequations}
where
\begin{align} 
	\label{eq:I}
	\I(f)
	:=
	\int_{0}^{\barx}\int_{0}^{\bary} \Rie_{ij}(x,y) f(x,y) dxdy.
\end{align}
To complete the proof, we next argue that each term in~\eqref{eq:Vxy1}--\eqref{eq:Vxy4} is a continuous function of $ (\chi,\psi) $.
For~\eqref{eq:Vxy4}, we indeed have $ \partial_x\Rie(x,0)\psi(y), \partial_y\Rie(0,y)\chi(x) \in \Generic $.
Consequently, the expression
\begin{align*}
	(\partial_x\Rie(x,0)\psi(y) + \generic_y)\cdot(\partial_y\Rie(0,y)\chi(x) + \generic_x)
\end{align*}
defines a continuous function of $ (x,y,\chi,\psi) \in \Rp^2 \times C(\Rp)^2 $.
Given this property, and referring to~\eqref{eq:I}, we see that the term in~\eqref{eq:Vxy4} is a continuous function of $ (\chi,\psi) $.
Turning to~\eqref{eq:Vxy3}, we note that the terms involve $ \chi' $ and $ \psi' $.
We integrate by parts in $ x $ and $ y $, \emph{respectively} for the first and second term in~\eqref{eq:Vxy3}.
This removes the derivatives on $ \chi $ and $ \psi $.
Further, $ \generic_x $ and $ \generic_y $ remain $ \Generic $-valued upon differentiating in $ x $ and $ y $, respectively.
From this, we see that the terms in \eqref{eq:Vxy3} are continuous functions of $ (\chi,\psi) $.
Moving onto~\eqref{eq:Vxy2}, we write
\begin{align*}
	\Rie(0,y)\chi'(x)\cdot\partial_y\Rie(0,y)\chi(x) &= \Rie(0,y)\partial_y\Rie(0,y)\cdot \tfrac{1}{2}\tfrac{d~}{dx} \chi(x)^2,
\\
	\partial_x\Rie(x,0)\psi(y)\cdot\Rie(x,0)\psi'(y) &= \Rie(x,0)\partial_x\Rie(x,0) \cdot \tfrac{1}{2}\tfrac{d~}{dy} \psi(x)^2.
\end{align*}
Given these expressions, integrating by parts in $ x $ and $ y $, \emph{respectively} for the first and second term in~\eqref{eq:Vxy2},
we conclude that the terms are continuous functions of $ (\chi,\psi) $.
Finally, for~\eqref{eq:Vxy1}, straightforward integration by parts in $ x $ and $ y $
verifies that the term is a continuous function of $ (\chi,\psi) $.
\end{proof}

Next we turn to the discrete analog of Lemma~\ref{lem:ibp}.
Recall that $ \rieI_{ij}(x,y) := \rieI(L\x_i-x, \y_i-y) \RieI(L\x_j-x,\y_j-y) $.
For $ \alpha=x,y $, set
\begin{align}
	\label{eq:Udx}
	\Vd_\alpha
	&:=
	L^{-2}\sum_{x,y\in\Z_{> 0}}
	\rieI_{ij}(x,y) \, (L\nabla_\alpha \barphi(x,y)) \, \barphi(x,y),
\\
	\label{eq:Udxy}
	\Vd_{xy}
	&:=
	L^{-2} \sum_{x,y\in\Z_{> 0}}
	\rieI_{ij}(x,y) \, (L\nabla_x \barphi(x,y)) \, (L\nabla_y \barphi(x,y)).
\end{align}
Recall the scaling notation and interpolation convention from Definition~\ref{defn:scale}.

\begin{lem}
\label{lem:sbp} 
Abusing notation, we write $ \V_{xy}(\barphi_L) := \V_{xy}(\barphi_L(\Cdot,0),\barphi_L(0,\Cdot)) $.
Then, as $ L\to\infty $,
\begin{align*}
	\sum_{\alpha=x,y} \big| \Vd_\alpha - \V_\alpha(\barphi_L) \big|
	+
	\big| \Vd_{xy} - \V_{xy}(\barphi_L) \big|
	\longrightarrow
	0.
\end{align*}
\end{lem}

\begin{proof}
We begin by bounding $ |\Vd_\alpha - \V_\alpha(\barphi_L)| $.
Take $ \alpha=x $ to simplify notation. The case $ \alpha=y $ follows exactly the same.
Let $ \tilde \Vd_x $ denote the analog of $ \Vd_x $ 
where the last factor $\bar\phi(x,y)$ in \eqref{eq:Udx} is replaced by $\bar\phi(x+1,y)$.
Using
\begin{align}
	\label{eq:nablasq}
	\nabla f(x) \cdot f(x+1) + f(x) \cdot \nabla f(x) = \nabla f^2(x)
\end{align}
for $ f(x) = \barphi(x,y) $, we have
\begin{align*}
	\frac12( \Vd_x + \tilde{\Vd}_x )
	=
	\frac{1}{2L^{2}} \sum_{x,y\in\Z_{> 0}}
	\rieI_{ij}(x,y)
	L\nabla_x \big(\barphi^2(x,y) \big).
\end{align*}	
Set $ \rie_{ij}(x,y) := \rie(L\x_i-x, L\y_i-y) \Rie(\x_j-x,\y_j-y) $,
and recall that $ \barx := \x_i\wedge \x_j $ and $ \bary := \y_i\wedge \y_j $.
Further, given the bounds on $ \rie $ from Lemma~\ref{lem:bound-R}
and the bound on $ \nabla_x\barphi $ from Lemma~\ref{lem:apribd},
we have $|\Vd_x-\tilde \Vd_x|\le c L^{-1} $, so
\begin{align}
	\label{eq:Vdx}
	\Vd_x
	=
	\frac{1}{2L^2} 
	\sum_{x,y\in\Z_{> 0}}
	\rieI_{ij}(x,y) \, L\nabla_x (\barphi(x,y)^2 )
	+
	r
	=
	\frac{1}{2L^2} 
	\sum_{x=1}^{L\barx} \sum_{y=1}^{L\bary}
	\rie_{ij}(x,y) \, L\nabla_x (\barphi(x,y)^2 )
	+
	r,	
\end{align}
for some remainder term $ r $ such that $ |r| \leq cL^{-1} $.
In~\eqref{eq:Vdx}, applying summation by parts 
\begin{align*}
	\sum_{i=a}^b f(i) \cdot \nabla g(i)
	=
	f(i-1) g(i) \big|_{i=a}^{i=b+1}
	-
	\sum_{i=a}^b \nabla f(i-1) \cdot g(i)
\end{align*}
in the variable $ x $ gives
\begin{align*}
	\Vd_x
	=
	\frac{1}{2L} \sum_{y=1}^{L \bary}
	\rie_{ij}(x-1,y) \, \barphi(x,y)^2  \,\Big|_{x=1}^{x=L\barx+1}
	-
	\frac{1}{2L^2}  \sum_{x=1}^{L\barx} \sum_{y=1}^{L\bary}
	L\nabla_x \rie_{ij}(x-1,y)
	\cdot
	\barphi(x,y)^2
	+ r.
\end{align*}
Given Lemmas~\ref{lem:bound-R}--\ref{lem:apribd},
within in the last expression,
replacing $\rie(x,y)$ with $\Rie(\tfrac{x}{L},\tfrac{y}{L})$, and 
$ L {\nabla_x} {\rie(x,y)} $ with $\partial_x\Rie(\tfrac{x}{L},\tfrac{y}{L}) $ 
only introduce errors that converges to zero as $ L\to\infty $.
This gives
\begin{align}
\label{eq:Vdxx}
	\Vd_x
	=
	\frac{1}{2L} \sum_{y=1}^{L \bary}
	\Rie_{ij}(\tfrac{x-1}{L},\tfrac{y}{L}) \, \barphi_L(\tfrac{x}{L},\tfrac{y}{L})^2  \,\Big|_{x=1}^{x=L\barx+1}
	-
	\frac{1}{2L^2}  \sum_{x=1}^{L\barx} \sum_{y=1}^{L\bary}
	L\nabla_x \Rie_{ij}(\tfrac{x-1}{L},\tfrac{y}{L})
	\cdot
	\barphi_L(\tfrac{x}{L},\tfrac{y}{L})^2
	+ r',
\end{align}
for some $ r' $ such that $ |r'|\to 0 $, where $\Rie_{ij}$ is defined in proof of Lemma~\ref{lem:ibp}.
Compare~\eqref{eq:VVx} and \eqref{eq:Vdxx}.
Since $ \{\phi_L\}_L \subset C([0,\barx]\times[0,\bary]) $ is equicontinuous,
and since $ \Rie $ is smooth,
in~\eqref{eq:Vdxx}, replacing sums with integrals and replacing $ L\nabla_x  $ with $ \partial_x $
only introduce errors that converges to zero as $ L\to\infty $.
From this we conclude $ |\Vd_x - \V_x(\barphi_L)| \to 0 $.

Turning to showing $ | \Vd_{xy} - \V_{xy}(\barphi_L)| \to 0 $,
we rewrite~\eqref{eq:barphi} in a way similar to~\eqref{eq:Phi:mild}
(note that $ \barphi=\phi $ along the axises $ \set{0}\times\Z_{\geq 0} $ and $ \Z_{\geq 0}\times\set{0} $):
\begin{align}
\label{eq:barphi:}
\begin{split}
	\barphi(x,y)
	=
	\barphi(0,0) \rie(x,y)
	&+
	L^{-1} \sum_{y'=1}^y \rieI(x,y-y') \big(L\nabla_{y'}\phi(0,y'-1)+ L(1-b_2) \barphi(0,y'-1)\big)
\\
	&+
 	L^{-1} \sum_{x'=1}^x \rieI(x-x',y) \big(L\nabla_{x'}\phi(x'-1,0) + L(1-b_1) \barphi(x'-1,0)\big).
	\end{split}
\end{align}
Define the discrete analog of $ \Generic $ (as in~\eqref{eq:Generic}):
\begin{align}
	\notag
	&\Genericd
	:= 
	\Big\{ 
		\sum \alpha_{L,\text{term}} \cdot \big(\text{term in~\eqref{eq:genericd}}\big) 
		: 
		\#\{ \alpha_{L,\text{term}} \neq 0 \}\leq c,
		\quad
		\lim_{L\to 0} \alpha_{L,\text{term}} \in \R
	\Big\},
\end{align}
where,
with $ k_i =(k_i,k'_i)\in\Z_{\geq 0}^2 $ being multi-indices with $ |k_i| \leq 3 $,
and with $ x_1\in\{0,x\} $, $ y_2\in\{0,y\} $, and $ j_i,j'_i,j''_i \in\{0,\pm1,\pm2\} $,
the terms read
\begin{align}
\label{eq:genericd}
\begin{split}
	&L^{-|k_1|}\nabla^{k_1} \rie(x-x_1+j_1,y) \chi(x_1+j'_1),
	\quad
	L^{-|k_2|}\nabla^{k_2} \rie(x,y-y_2+j_2) \psi(y_2+j'_2),		
\\
	&
	L^{-1} \sum_{x'=1}^{x} L^{-|k_3|}\nabla^{k_3} \rie(x-x'+j_3,y+j'_3) \chi(x'+j''_3),
\\
	&
	L^{-1} \sum_{y'=1}^{y} L^{-|k_3|}\nabla^{k_3} \rie(x+j_4,y-y'+j'_4) \psi(y'+j''_4).
\end{split}
\end{align}

Under the preceding setup,
we perform procedures analogous to those leading up to~\eqref{eq:Phix}--\eqref{eq:Phiy},
with~\eqref{eq:Vdxx} in place of~\eqref{eq:Phi:mild},
$\Vd_{xy}$ in place of $\V_{xy}$,
$\rie$ and $L \nabla_x \rie$ in place of $\Rie$ and $\partial_x \Rie$, and
$(\barphi(\Cdot,0),\barphi(0,\Cdot))$  in place of $(\chi,\psi)$.
This gives 
\begin{align}
	\label{eq:barphix}
	L\nabla_x \barphi(x,y)
	&=
	\rie(0,y)\cdot L\nabla_x \barphi(x,0) + L\nabla_x \rie(x,0)\cdot\barphi(0,y) + \genericd_y,
\\	
	\label{eq:barphiy}
	L\nabla_y \barphi(x,y)
	&=
	\rie(x,0)\cdot L\nabla_y \barphi(0,y) + L\nabla_y \rie(0,y) \cdot\barphi(x,0) + \genericd_x,
\end{align}
for some $ \genericd_\alpha $ such that $ \genericd_\alpha, L\nabla_\alpha\genericd_\alpha \in \Genericd $.
For our purpose it is more convineint to change $ \barphi(0,y) \mapsto \frac12(\barphi(0,y)+\barphi(0,y+1)) $
and $ \barphi(x,0) \mapsto \frac12(\barphi(x,0)+\barphi(x+1,0)) $ in~\eqref{eq:barphix}--\eqref{eq:barphiy}.
To this end, using the bounds on $ L \nabla_\alpha \rie $ from Lemma~\ref{lem:bound-R}
and the bound on $ \nabla_\alpha \barphi $ from Lemma~\ref{lem:apribd},
we write
\begin{align}
	\tag{\ref*{eq:barphix}'}
	\label{eq:barphix:}
	L\nabla_x \barphi(x,y)
	&=
	\rie(0,y)\cdot L\nabla_x \barphi(x,0) + L\nabla_x \rie(x,0)\cdot\tfrac12\big( \barphi(0,y)+\barphi(0,y+1) \big) + \genericd_y +r_1,
\\	
	\tag{\ref*{eq:barphiy}'}
	\label{eq:barphiy:}
	L\nabla_y \barphi(x,y)
	&=
	\rie(x,0)\cdot L\nabla_y \barphi(0,y) + L\nabla_y \rie(0,y) \cdot \tfrac12\big( \barphi(x,0)+ \barphi(x+1,0) \big) + \genericd_x+r_2,
\end{align}
for some $ r_1,r_2 $ such that $ |r_1|,|r_2| \leq cL^{-1} $.

Inserting~\eqref{eq:barphix:}--\eqref{eq:barphiy:} into~\eqref{eq:Udxy} gives
\begin{subequations}
\begin{align}
	\label{eq:vxy1}
	\Vd_{xy}(\chi,\psi) 
	=& 
	\Id\Big( \rie(0,y)L\nabla_x \barphi(x,0)\cdot\rie(x,0)L\nabla_y \barphi(0,y) \Big)
\\
	\begin{split}
	\label{eq:vxy2}
	&
	+\Id\Big( \rie(0,y)\cdot L\nabla_x \barphi(x,0)\cdot L\nabla_y\rie(0,y)\cdot \tfrac12\big( \barphi(x,0)+ \barphi(x+1,0) \big)  \Big) 
\\
	&
	\quad\quad+\Id\Big( L\nabla_x\rie(x,0)\cdot\tfrac12\big( \barphi(0,y)+\barphi(0,y+1) \big)\cdot\rie(x,0)\cdot L\nabla_y \barphi(0,y)  \Big)	
	\end{split}
\\
	&
	\label{eq:vxy3}	
	+\Id\Big( \rie(0,y)L\nabla_x \barphi(x,0)\cdot\genericd_x \Big) + \Id\Big( \genericd_y\cdot\rie(x,0)L\nabla_y \barphi(0,y)  \Big)	
\\
	\begin{split}
	\label{eq:vxy4}	
	&	
	+\Id\Big( \Big(L\nabla_x\rie(x,0)\cdot\tfrac12\big( \barphi(0,y)+\barphi(0,y+1) \big) + \genericd_y \Big)
\\
	&
	\hphantom{+\Id\Big(\quad }
	\cdot \Big(L\nabla_y\rie(0,y)\cdot \tfrac12\big( \barphi(x,0)+\barphi(x+1,0) + \genericd_x)  \Big)\Big)
	\end{split}
\\
	\notag
	&+r',
\end{align}
\end{subequations}
where
$
	\Id(f)
	:=
	L^{-2} \sum_{x=1}^{L\barx}\sum_{y=1}^{L\bary} \rie_{ij}(x,y) f(x,y).
$
Here $ r' $ collects all the terms that involve $ r_1 $ and $ r_2 $ from the expansion.
Given the bounds from Lemmas~\ref{lem:bound-R} and \ref{lem:apribd},
we indeed have $ |r'| \leq c L^{-1} $.
Also, using~\eqref{eq:nablasq} for $ f(x)=\barphi(x,0) $ and for $ f(y)=\barphi(0,y) $,
we rewrite the terms in~\eqref{eq:vxy2} as
\begin{subequations}
\begin{align}
	\tag{\ref*{eq:vxy2}'}
	\label{eq:vxy2:}
	\Id\Big( \rie(0,y)\cdot L\nabla_y\rie(0,y)\cdot L\nabla_x \big( \barphi(x,0)^2 \big)  \Big) 
	+\Id\Big( L\nabla_x\rie(x,0)\cdot\rie(x,0)\cdot L\nabla_y \big( \barphi(0,y)^2 \big)  \Big).
\end{align}
\end{subequations}
Recall that, in the proof of Lemma~\ref{lem:ibp},
we processed the terms in~\eqref{eq:Vxy1}--\eqref{eq:Vxy4} through integration by parts
so that the resulting expressions do not involve derivatives of $ \chi $ or of $ \psi $.
Here, similarly processing \eqref{eq:vxy1}, \eqref{eq:vxy2:}, and \eqref{eq:vxy3}--\eqref{eq:vxy4} (via summation by parts)
gives expressions that do not involve discrete gradients of $ \barphi(x,0) $ or of $ \barphi(0,y) $.
Given that $ \{\phi_L(x,y)\}_L \subset C([0,\barx]\times[0,\bary]) $ is equicontinuous, 
and given the bounds from Lemma~\ref{lem:bound-R},
within the \emph{processed} expressions of~\eqref{eq:vxy1}, \eqref{eq:vxy2:}, and \eqref{eq:vxy3}--\eqref{eq:vxy4}, 
replacing $ \rie $ and $ L^{|k|}\nabla^k\rie $ with $ \Rie $ and $ \partial^k \Rie $
and replacing the sums with integrals only cause errors that converge to zero as $ L\to\infty $.
From this we conclude $ | \Vd_{xy} - \V_{xy}(\barphi_L) | \to 0 $.
\end{proof}

Based on Lemmas~\ref{lem:ibp}--\ref{lem:sbp}, we finish the proof of Proposition~\ref{prop:qv}\ref{prop:qv:S1}.
With $ S_1 $ defined in \eqref{eq:S1}, 
referring to~\eqref{eq:D}, \eqref{eq:Dd}, \eqref{eq:U}--\eqref{eq:Udx}, and \eqref{eq:Udx}--\eqref{eq:Udxy},
we decompose $ S_1 = \sum_{\beta=xy,x,y}\sum_{i=1,2} S^i_{1\beta} $, where
\begin{align*}
	&
	S^1_{1xy} := (\gamma_{xy}-\beta_1+\beta_2) \Vd_{xy},&
	&
	S^1_{1x} := (\gamma_{x}-\beta_2 (\beta_2 -\beta_1)) \Vd_{x},&
	&
	S^1_{1y} := (\gamma_{y}+\beta_1 (\beta_2 -\beta_1)) \Vd_{y},	
\\
	&
	S^2_{1xy} :=  (\beta_1+\beta_2) \big( \Vd_{xy}-\V_{xy}(\Phi) \big),&
	&
	S^2_{1x} :=  \beta_2 (\beta_2 -\beta_1) \big( \Vd_{x}-\V_{x}(\Phi) \big),&
	&
	S^2_{1y} :=  - \beta_1 (\beta_2 -\beta_1)  \big( \Vd_{y}-\V_{y}(\Phi) \big).
\end{align*}
For $ S^1_{1\beta} $, $ \beta=xy,x,y $, it is readily checked from Lemma~\ref{lem:apribd} that $ |\Vd_\beta| \leq c $.
From~\eqref{eq:gamma}, we have that $ \gamma_{xy}\to(\beta_1+\beta_2) $,
$ \gamma_{x}\to \beta_2 (\beta_2 -\beta_1) $, and $ \gamma_{y}\to -\beta_1 (\beta_2 -\beta_1) $.
Hence $ S^1_{1\beta} \to 0 $.
As for $ S^2_{1\beta} $, $ \beta=xy,x,y $, further decompose
$ \Vd_{\beta}-\V_{\beta}(\Phi) = (\Vd_\beta - \V_\beta(\barphi_L)) + (\V_\beta(\barphi_L) - \V_{\beta}(\Phi)) $.
Using Lemmas~\ref{lem:ibp}--\ref{lem:sbp} and~\eqref{eq:barphi:cnvg} to bound the respective terms,
we conclude $ S^2_{1\beta}\to 0 $.

\section{Proof of Proposition~\ref{prop:qv}\ref{prop:qv:S2}}
\label{sec:S2}

The proof begins by deriving an integral representation for $ L\nabla_x\phi $ and $ L\nabla_y\phi $.
To this end, rewrite~\eqref{eq:phi} as
\begin{align*}
	\phi(x,y) = \barphi(x,y)+
  	\sum_{x'\in(0,x]}\sum_{y'\in(0,y]} \rie(x-x',y-y') \dnoise(x',y').
\end{align*}
Take discrete derivatives on both sides to get
\begin{align}
	\label{eq:phix}
	&L\nabla_x \phi(x,y)
	= 
	L\nabla_x \barphi(x,y)
	+
	\bk_x(x,y)
	+
	\bdy_x(x,y),&
	&
	L\nabla_y \phi(x,y)
	= 
	L\nabla_y \barphi(x,y)
	+
	\bk_y(x,y)
	+
	\bdy_y(x,y),
\end{align}
where
\begin{align}
	\label{eq:bk}
	\bk_\alpha(x,y)
	&:=
	\sum_{x',y'\in\Z_{> 0}} L\nabla_\alpha \rieI(x-x',y-y') \dnoise(x',y'),
\\
	\label{eq:bdy}
	\bdy_x(x,y)
	&:=
	L\sum_{y'\in\Z_{> 0}} \rieI(0,y-y') \dnoise(x+1,y'),
	\quad
	\bdy_y(x,y)
	:=
	L\sum_{x'\in\Z_{> 0}} \rieI(x-x',0) \dnoise(x',y+1).
\end{align}
%

\begin{lem}
\label{lem:ptmom}
For any fixed $ a\in[1,\infty) $ and $ f:\Z_{\geq 0}^2 \to \R $, $ \alpha=x,y $,
we have
\begin{align*}
	\sup_{x,y \in[0,aL]\cap\Z} \Ex[(\bk_\alpha(x,y))^2]
	\leq
	c(a) L^{-1},
	\quad\quad
	\sup_{x,y\in[0,aL]\cap\Z} \Ex[(\bdy_\alpha(x,y))^2]
	\leq
	c(a).
\end{align*}
\end{lem}
\begin{proof}
For simpler notation, throughout the proof we write $ c=c(a) $,

Calculate the second moment of $ \bk_\alpha(x,y) $~\eqref{eq:bk}.
By Lemma~\ref{lem:bound-R}, the variables $ \dnoise(x,y) $, $ x,y\in\Z_{>0}^2 $ are uncorrelated, so
\begin{align}
	\label{eq:bk:contraction}
	\Ex[(\bk_\alpha(x,y))^2]
	=
	\sum_{x',y'\in\Z_{>0}} 
	\big( L\nabla_\alpha\rieI(x-x',y-y') \big)^2
	\Ex\big[\dnoise(x',y')^2 \big].
\end{align}
By Lemma~\ref{lem:bound-R}, the term $ L\nabla_\alpha\rieI(x-x',y-y') $ is bounded by $ c $.
With $ x',y' \in [0,aL]\cap\Z $, the number of terms within the sum is $ \leq c L^2 $.
By Lemma~\ref{lem:cond-mom-xi}, the $ \Ex[\dnoise(x,y)^2 ]\leq cL^{-3} $.
From these discussions, we concludes the desired bound for $ \bk_\alpha $.

We now turn to bounding $ \bdy_\alpha $. Take $ \alpha=x $ to simplify notation.
Following the same argument for obtaining~\eqref{eq:bdy:contraction}, here we have
\begin{align}
	\label{eq:bdy:contraction}
	\Ex[(\bdy_x(x,y))^2]
	=
	 L^2\sum_{y'\in\Z_{>0}} \big( \rieI(0,y-y') \big)^2\Ex\big[\dnoise(x+1,y')^2 \big].
\end{align}
By Lemma~\ref{lem:bound-R}, the term $ \rieI(0,y-y') $ is bounded by $ c $.
With $ y' \in [0,aL]\cap\Z $, the number of terms within the sum is $ \leq c L $.
By Lemma~\ref{lem:cond-mom-xi}, the $ \Ex[\dnoise(x,y)^2 ]\leq cL^{-3} $.
From these discussions, we concludes the desired bound for $ \bdy_\alpha $.
\end{proof}


Having established Lemma~\ref{lem:ptmom}, we now proceed to bounding $ S_2 $.
To simplify notation, set
\begin{align*}
	\quad\quad
	&L^{-2}\sum_{s=1}^{c_*L^2} f(x(s),y(s))
	:=
	\Av(f).
\end{align*}
First, recall from \eqref{eq:S2} and \eqref{eq:Dd} that
$ S_2 $ involves the term $ \phi(x,y)  L\nabla_x \phi(x,y) $ and $ \phi(x,y) L\nabla_y \phi(x,y) $ via $ \Dd(x,y;\phi) $.
From Theorem~\ref{thm:lln} and \eqref{eq:barphi:cnvg}, we have that
$
	\norm{\phi_L-\barphi_L}_{C(\Rp^2)} \to 0,
$
as $ L\to\infty$.
From this, together with the bound \eqref{eq:phi:apribd} on $ \phi $ and the bounds on $ \rie $ from Lemma~\ref{lem:bound-R},
we see that
\begin{align*}
	\Av\Big(
	\rieI_{ij}\cdot
	\big( 
		\gamma_x \phi L\nabla_x \phi -\gamma_x \barphi L\nabla_x \phi
		+
		\gamma_y \phi  L\nabla_y \phi - \gamma_y\barphi L\nabla_y \phi
	\big)
	\Big)
	\longrightarrow_\text{P}0.
\end{align*}
Granted this, instead of showing $ S_2 \to_\text{P} 0 $, it suffices to show $ \hat{S}_2 \to_\text{P} 0 $,
where
\begin{align}
	\label{eq:hatS2}
	\hat{S}_2
	&:=
	L^{-2}\sum_{s=1}^{c_*L^2} \rieI_{ij}(x(s),y(s))\cdot \big( \Ddd(x(s),y(s);\barphi,\phi) - \Dd(x(s),y(s);\barphi) \big),
\\
	\label{eq:Ddd}
	\Ddd(x,y;\barphi,\phi)  &:=
   \gamma_{xy} \cdot
   	L \nabla_x \phi (x,y) \cdot L \nabla_y \phi(x,y)
    +\gamma_{x}\cdot \barphi(x,y) \cdot L\nabla_x \phi(x,y)
   + \gamma_{y}\cdot \barphi(x,y) \cdot L\nabla_y \phi(x,y).
\end{align}

Now, insert the expressions~\eqref{eq:phix} for $ L\nabla_x \phi(x,y) $ and $ L\nabla_y \phi(x,y) $ into the r.h.s.\ of~\eqref{eq:Ddd},
and plug the result into~\eqref{eq:hatS2}.
Expand the result accordingly, 
we have that, for some bounded, deterministic $ f_{\bk\bk},f_{\bk\bdy},g_{\bk\bdy},\ldots:\Z_{>0}^2 \to \R $,
\begin{align*}
	\hat{S}_2 
	&= 
	\Av\big( f_{\bk\bk}(\bk_x\bk_y) \big)
	+
	\Av\big( f_{\bk}\bk_x+g_{\bk}\bk_y) \big)
	+
	\Av\big( f_{\bk\bdy}\bk_x\bdy_y+g_{\bk\bdy}\bk_y\bdy_x\big)	
	+
	\Av\big( f_{\bdy}\bdy_x+g_\bdy\bdy_y\big)	
	+
	\Av\big( f_{\bdy\bdy}\bdy_x\bdy_y\big)
\\
	&:=	K_{\bk\bk} + K_{\bk} + K_{\bk\bdy}+K_{\bdy} + K_{\bdy\bdy}.
\end{align*}
Write $ \norm{\,\Cdot\,}_k := (\Ex[|\,\Cdot\,|^k])^{1/k} $ for the $ k $-th norm.
By triangle inequality and Cauchy--Schwarz inequality, for $ p \geq 1 $,
\begin{align*}
	\norm{ K_{\bk\bk} }_{p}
	&\leq
	\Av\big( |f_{\bk\bk}|\cdot(\norm{\bk_x}_{2p}\norm{\bk_y}_{2p}) \big),
\\
	\norm{ K_{\bk} }_{p}
	&\leq
	\Av\big( (|f_{\bk}|+|g_{\bk}|)\cdot(\norm{\bk_x}_p + \norm{\bk_y}_p) \big),
\\
	\norm{ K_{\bk\bdy} }_{p}
	&\leq
	\Av\big( (|f_{\bk\bdy}|+|g_{\bk\bdy}|)\cdot(\norm{\bk_x}_{2p}\norm{\bdy_y}_{2p} + \norm{\bk_y}_{2p}\norm{\bdy_x}_{2p}) \big).
\end{align*}
Given that $ f_{\bk\bk} $ and $ f_{\bk\bdy} $ are bounded,
apply Lemma~\ref{lem:ptmom} gives $ \norm{ K_{\bk\bk} }_{1}, \norm{ K_{\bk} }_{2}, \norm{ K_{\bk\bdy} }_{1} \to 0 $.

It now remains to show that $ K_{\bdy}, K_{\bdy\bdy} \to_\text{P} 0 $.
From~\eqref{eq:bdy:contraction} and \eqref{eq:condVar-xi},
it is not hard to check that $ \Ex[\bdy_x(x,y)^2] \not\to 0 $ (so our bound in Lemma~\ref{lem:ptmom} is sharp).
Given this situation, unlike in the preceding, 
here we cannot apply triangle inequality to pass $ \norm{ \,\Cdot\, }_{1} $ into the sum $ \Av $.
Instead, we need to exploit the averaging effect of $ \Av $.
This is done in the following Lemma, which completes the proof.

\begin{lem}
Given deterministic $ f:\Z_{>0}^2\to\R $ and $ a<\infty $, we have that
\begin{align}
	\label{eq:bdydecay}
	\Ex\Big[ \Big( L^{-2}\sum_{x,y\in[0,aL]\cap\Z} f(x,y) \bdy_{\alpha}(x,y) \Big)^2 \Big]
	&\leq
	c(a)L^{-1}\norm{f}_{L^\infty(\Z_{>0}^2)}^2,
	\quad
	\alpha=x,y,
\\
	\label{eq:bdy2decay}
	\Ex\Big[ \Big( L^{-2}\sum_{x,y\in[0,aL]\cap\Z} f(x,y) \bdy_{x}(x,y)\bdy_{y}(x,y) \Big)^2 \Big]
	&\leq
	c(a)L^{-1}\norm{f}_{L^\infty(\Z_{>0}^2)}^2,
\end{align}
In particular, $ \norm{K_{\bk}}_2 + \norm{K_{\bdy\bdy}}_2 \to 0 $.
\end{lem}

\begin{proof}
Fixing $ a\in[1,\infty) $ and $ f:\Z_{\geq 0}^2 \to \R $. 
To simplify notation, throughout the proof we write $ c=c(a) $,
and, always assume (without explicitly stating) that variables $ x,y,x_1 $, etc., are in $ [0,aL]\cap\Z $.

We begin with the bound~\eqref{eq:bdydecay}.
Take $ \alpha=x $ to simplifyy notation.
Calculate the l.h.s.\ of~\eqref{eq:bdydecay} from~\eqref{eq:bdy}. 
By Lemma~\ref{lem:bound-R}, the variables $ \dnoise(x,y) $, $ x,y\in\Z_{>0}^2 $ are uncorrelated, so
\begin{align*}
	\text{l.h.s.\ of }\eqref{eq:bdydecay}
	&=
	L^{-2} \sum_{(x_1,y_1),(x_2,y_2)} \sum_{y'_1,y'_2}
	\Big(\prod_{i=1}^2 f(x_i,y_i) \rieI(0,y_i-y'_i) \Big)\Ex\big[\dnoise(x_1+1,y'_1)\dnoise(x_2+1,y'_2)\big]
\\
	&=
	L^{-2} \sum_{(x_1,y_1),(x_2,y_2),x_1=x_2} \sum_{y}
	\Big( \prod_{i=1}^2 f(x_i,y_i) \rieI(0,y_i-y) \Big) \Ex\big[\dnoise(x_1+1,y)^2\big].
\end{align*}
By Lemma~\ref{lem:bound-R}, the Riemann function $ \rieI $ is bounded,
and by Lemma~\ref{lem:cond-mom-xi}, $ \Ex[\dnoise(x_1+1,y)^2] \leq cL^{-3} $.
With $ x_i,y_i,y \in [0,aL]\cap\Z $, the number of terms within the sum is $ \leq cL^{3+1} $.
From this we conclude
\begin{align*}
	\text{l.h.s.\ of }\eqref{eq:bdydecay}
	\leq
	cL^{-2} L^{3+1}\norm{f}_{L^\infty(\Z_{>0}^2)}^2 L^{-3}
	\leq
	cL^{-{1}}\norm{f}_{L^\infty(\Z_{>0}^2)}^2.	
\end{align*}

We now move onto~\eqref{eq:bdy2decay}.
Similarly to the preceding, we calculate
\begin{align}
\begin{split}
	\label{eq:bdy:contration4}
	\text{l.h.s.\ of }\eqref{eq:bdy2decay}
	=
	\sum_{(x_1,y_1),(x_2,y_2)} \sum_{x'_1,x'_2,y'_1,y'_2}
	&\Big(\prod_{i=1}^2 f(x_i,y_i) \rieI(0,y_i-y'_i) \rieI(x_i-x'_i,0) \Big)
\\	
	&
	\cdot
	\Ex\Big[\prod_{i=1}^2\dnoise(x_i+1,y'_i)\prod_{i=1}^2\dnoise(x_i',y_i+1)\Big].
\end{split}
\end{align}
To bound the r.h.s.\ of~\eqref{eq:bdy:contration4}, we proceed by discussing the relative location of the following four points where $\dnoise$ is evaluated:
\begin{align*}
(x(s_1),y(s_1)):=(x_1+1,y'_1),
&\qquad
(x(s_2),y(s_2)):=(x_2+1,y'_2),
\\
(x(s_3),y(s_3)):=(x_3',y_3+1),
&\qquad
(x(s_4),y(s_4)):=(x_4',y_4+1).
\end{align*}
Here, $s_i \in\Z_{> 0}  $ denotes the order of the point under the linear ordering~\eqref{eq:linord}.
For example, if $ (x_2+1,y'_2)=(2,2) $, $ s_2=3 $.
Let $ s_{*}=\max\{s_1,\ldots,s_4\} $ denote the maximal order among the four points,
and let $ \filtnoi(t) := \sigma(\dnoise(x(1),y(1)),\ldots,\dnoise(x(t),y(t))) $
denote the canonical filtration of $\dnoise(x,y) $ under the linear ordering~\eqref{eq:linord}.
\begin{enumerate}
\item	\label{bk:1-}
	The point  $(x(s_*),y(s_*))$    
	is separated from the other three points.\\
	In this case, first take conditional expectation $ \Ex[\,\Cdot\,|\filtnoi(s_*-1)] $,
	with the aid of~\eqref{eq:condE-xi}, we have
	\begin{align*}
		\Ex\Big[ \!\!\!\! \prod_{s\in\{s_1,\cdots,s_4\}} \!\!\!\! \dnoise(x(s),y(s)) \Big]
		=
		\Ex\Big[\prod_{s\neq s_*}\dnoise(x(s),y(s)) \, \Ex\big[\dnoise(x(s_*),y(s_*))|\filtnoi(s_* -1)\big] \Big] = 0.
	\end{align*}
\item	\label{bk:211} 
	The point $ (x(s_*),y(s_*)) $ is identical with another point, 
	and the other two points are separated.\\
	Take $ s_1=s_2>s_3>s_4 $ to simplify notation, and other permutations follows exactly the same.
	In this case, take conditional expectation $ \Ex[\,\Cdot\,|\filtnoi(s_{1}-1)] $, $ \Ex[\,\Cdot\,|\filtnoi(s_{3}-1)] $,
	and $ \Ex[\,\Cdot\,|\filtnoi(s_{4}-1)] $ in order,
	using Lemma~\ref{lem:cond-mom-xi} for $ k=2,1,1 $, respectively, we have
	\begin{align*}
		\Ex\Big[ \!\!\!\! \prod_{s\in\{s_1,\cdots,s_4\}} \!\!\!\! \dnoise(x(s),y(s)) \Big]
		\leq c L^{-2-1} L^{-1-1} L^{-1-1} = cL^{-7}.
	\end{align*}
\item \label{bk:22} 
The point $ (x(s_*),y(s_*)) $ is identical with another point, 
	and the other two points are identical.  \\  
	Take $ s_1=s_2>s_3=s_4 $ to simplify notation, and other permutations follows exactly the same.
	In this case, take conditional expectation $ \Ex[\,\Cdot\,|\filtnoi(s_{1}-1)] $, $ \Ex[\,\Cdot\,|\filtnoi(s_{3}-1)] $
	in order,
	using Lemma~\ref{lem:cond-mom-xi} for $ k=2,2 $, respectively, we have
	\begin{align*}
		\Ex\Big[ \!\!\!\! \prod_{s\in\{s_1,\cdots,s_4\}} \!\!\!\! \dnoise(x(s),y(s)) \Big]
		 \leq c L^{-1-2} L^{-1-2} = cL^{-6}.
	\end{align*}
\item \label{bk:31}
The point $ (x(s_*),y(s_*)) $ is identical with two other points, 
	and the fourth point is separated.\\
	Take $ s_1=s_2=s_3>s_4 $ to simplify notation, and other permutations follows exactly the same.
	In this case, take conditional expectation $ \Ex[\,\Cdot\,|\filtnoi(s_{1}-1)] $, $ \Ex[\,\Cdot\,|\filtnoi(s_{4}-1)] $
 in order,
	using Lemma~\ref{lem:cond-mom-xi} for $ k=3,1 $, respectively, we have
	\begin{align*}
		\Ex\Big[ \!\!\!\! \prod_{s\in\{s_1,\cdots,s_4\}} \!\!\!\! \dnoise(x(s),y(s)) \Big]
		\leq c L^{-1-3} L^{-1-1} = cL^{-6}.
	\end{align*}
\item \label{bk:4} 
	All four points are together.\\
	Using Lemma~\ref{lem:cond-mom-xi} for $ k=4 $ gives
	\begin{align*}
	\Ex\Big[ \!\!\!\! \prod_{s\in\{s_1,\cdots,s_4\}} \!\!\!\! \dnoise(x(s),y(s)) \Big]
		\leq c L^{-1-4} = cL^{-5}.
	\end{align*}
\end{enumerate}
\noindent
Now, with $ x_i,y_i \in [0,aL]\cap\Z $, the number of terms 
within the sum in~\eqref{eq:bdy:contration4} is of order $ L^{8} $.
Each contraction of points reduce the number of terms by $ L^{-2} $.
For example, the number of terms corresponding the case \eqref{bk:211} is $ \leq cL^{8-2} $,
because $ (x(s_*),y(s_*)) $  
being joined once amounts to contracting one point.
Following this line of reasoning, the number of terms 
within each cases~\eqref{bk:211}--\eqref{bk:4}
are bounded by $ cL^{6} $, $ cL^{4} $, $ cL^{4} $, $ cL^{2} $, respectively.
From these discussions, we bound the r.h.s.\ of~\eqref{eq:bdy:contration4} by
\begin{align*}
	\text{l.h.s.\ of }\eqref{eq:bdydecay}
	\leq
 cL^{6-7}+cL^{4-6}+cL^{4-6}+cL^{2-5}  \leq cL^{-1}.
\end{align*}
This concludes the proof.
%
\end{proof}

\bibliographystyle{alphaabbr}
\bibliography{Reference}

\begin{thebibliography}{BBCW17}

\bibitem[AB16]{AB16}
A.~Aggarwal and A.~Borodin.
\newblock Phase transitions in the {ASEP} and stochastic six-vertex model.
\newblock {\em arXiv:1607.08684}, 2016.

\bibitem[Agg16]{A16aa}
A.~Aggarwal.
\newblock Current fluctuations of the stationary {ASEP} and six-vertex model.
\newblock {\em Duke Math. J.}, 167:269--384, 2016.

\bibitem[BBCW17]{BBCW}
G.~Barraquand, A.~Borodin, I.~Corwin, and M.~Wheeler.
\newblock Stochastic six-vertex model in a half-quadrant and half-line open
  {ASEP}.
\newblock {\em arXiv:1704.04309}, 2017.

\bibitem[BCG16]{BCG2016}
A.~Borodin, I.~Corwin, and V.~Gorin.
\newblock Stochastic six-vertex model.
\newblock {\em Duke Math. J.}, 165(3):563--624, 2016.

\bibitem[BG18]{borodin18}
A.~Borodin and V.~Gorin.
\newblock A stochastic telegraph equation from the six-vertex model.
\newblock {\em arXiv:1803.09137}, 2018.

\bibitem[BO17]{BO16}
A.~Borodin and G.~Olshanski.
\newblock The {ASEP} and determinantal point processes.
\newblock {\em Commun. Math. Phys.}, 353:853--903, 2017.

\bibitem[BP16]{BP16a}
A.~Borodin and L.~Petrov.
\newblock Higher spin six vertex model and symmetric rational functions.
\newblock {\em Selecta Math.}, 2016.

\bibitem[CGST18]{CGST18}
I.~Corwin, P.~Ghosal, H.~Shen, and L.-C. Tsai.
\newblock Stochastic pde limit of the six vertex model.
\newblock {\em arXiv preprint arXiv:1803.08120}, 2018.

\bibitem[CP16]{Corwin2016}
I.~Corwin and L.~Petrov.
\newblock Stochastic higher spin vertex models on the line.
\newblock {\em Commun. Math. Phys.}, 343(2):651--700, 2016.

\bibitem[CT17]{CT15}
I.~Corwin and L.-C. Tsai.
\newblock {KPZ} equation limit of higher-spin exclusion processes.
\newblock {\em Ann. Probab.}, 45:1771--1798, 2017.

\bibitem[GS92]{GS92}
L.-H. Gwa and H.~Spohn.
\newblock Six-vertex model, roughened surfaces, and an asymmetric spin
  {H}amiltonian.
\newblock {\em Phys. Rev. Lett.}, 68:725--728, 1992.

\bibitem[HH14]{hall2014}
P.~Hall and C.~C. Heyde.
\newblock {\em Martingale limit theory and its application}.
\newblock Academic press, 2014.

\bibitem[RS16]{ResSri16b}
N.~Reshetikhin and A.~Sridhar.
\newblock Limit shapes of the stochastic six vertex model.
\newblock {\em arXiv:1609.01756}, 2016.

\end{thebibliography}

\end{document}